%% file: internal-enriched-cats.tex
\pdfoutput=1

\documentclass[paper=a4, pagesize, fontsize=12pt, twoside=true, DIV=calc, bibliography=totoc, final, version=last]{scrartcl}

\usepackage[T1]{fontenc}
\usepackage[utf8]{inputenc}
\usepackage[USenglish]{babel}
\usepackage{lmodern}

\usepackage[bibstyle=alphabetic, citestyle=alphabetic, backend=biber]{biblatex} %
\usepackage{csquotes} %

\usepackage[automark]{scrlayer-scrpage}
\usepackage{mathtools} %
	\mathtoolsset{mathic} %
\usepackage{amsfonts}
\usepackage{amsthm}
\usepackage{amssymb}
\usepackage{bbold} %
\usepackage{mathdots} %
\usepackage{braket} %
\usepackage{xfrac} %
\UseCollection{xfrac}{plainmath}
\usepackage{tikz-cd}
	\usetikzlibrary{arrows}
\KOMAoptions{captions=heading}

\usepackage{microtype}

\usepackage{hyperxmp} %
\usepackage[hidelinks, unicode=true, linktoc=all, pdfencoding=unicode]{hyperref} %
\hypersetup{
	pdfauthor={Enrico Ghiorzi},
	pdftitle={Internal enriched categories},
	pdfsubject={Category Theory},
	pdfcopyright={Copyright \copyright\ 2020 Enrico Ghiorzi},
	pdfkeywords={category theory, internal categories, enriched categories},
	pdftype={Text}
}
\usepackage{bookmark} %
\usepackage{cleveref} %

\recalctypearea

\input{math.tex}

\NewDocumentCommand{\A} {} {\IntCatStyle{A}}
\NewDocumentCommand{\ExtA} {} {\CatStyle{A}}
\ProvideDocumentCommand{\B} {} {\IntCatStyle{B}}
\RenewDocumentCommand{\C} {} {\IntCatStyle{C}}
\NewDocumentCommand{\IntC} {} {\IntCatStyle{C}}
\NewDocumentCommand{\IntD} {} {\IntCatStyle{D}}
\NewDocumentCommand{\D} {} {\IntCatStyle{D}}
\NewDocumentCommand{\F} {} {\symcal{F}}

\NewDocumentCommand{\I} {} {\CatStyle{I}}
\ProvideDocumentCommand{\M} {} {\CatStyle{M}}
\NewDocumentCommand{\E} {} {\CatStyle{E}}
\NewDocumentCommand{\ExtV} {} {\CatStyle{V}}
\NewDocumentCommand{\IntV} {} {\IntCatStyle{V}}
\NewDocumentCommand{\MCV} {} {\MultiCatStyle{V}}
\NewDocumentCommand{\MCX} {} {\MultiCatStyle{X}}
\NewDocumentCommand{\MCY} {} {\MultiCatStyle{Y}}
\NewDocumentCommand{\X} {} {\IntCatStyle{X}}
\NewDocumentCommand{\Y} {} {\IntCatStyle{Y}}
\NewDocumentCommand{\W} {} {\IntCatStyle{W}}
\NewDocumentCommand{\IntW} {} {\IntCatStyle{W}}
\NewDocumentCommand{\IdxC} {} {\CatStyle{C}}
\NewDocumentCommand{\IdxW} {} {\CatStyle{W}}
\NewDocumentCommand{\Z} {} {\IntCatStyle{Z}}

\begin{document}

\hyphenation{cat-e-go-ry cat-e-go-ries multi-cat-e-go-ry multi-cat-e-go-ries}
\hyphenation{car-te-sian car-te-sian-ness}
\hyphenation{mo-nad mo-nads mo-na-dic group-oid group-oids}
\hyphenation{pseu-do-func-tor}
	
\title{Internal enriched categories}
\author{Enrico Ghiorzi\thanks{Research funded by Cambridge Trust and EPSRC.}}
\maketitle

\abstract{
We introduce the theory of enrichment over an internal monoidal category as a common generalization of both the standard theories of enriched and internal categories.
The aim of the paper is to justify and contextualize the new notion by comparing it to other known generalizations of enrichment: namely, those for indexed categories and for generalized multicategories.
It turns out that both of these notions are closely related to internal enrichment and, as a corollary, to each other.
}

\tableofcontents

\addsec{Introduction}

Size issues prevent many fundamental results in category theory from being straightforward.
At the core, the problem is that categories, even large ones, are said to be complete if they have limits for merely all \emph{small} diagrams.
Consider the following examples:
\begin{itemize}
\item The adjoint functor theorem states that a functor \(\CatStyle{C} \to \CatStyle{D}\) which preserves limits and whose domain category \(\CatStyle{C}\) is complete is a right adjoint.
Its left adjoint is defined by using the limits of certain diagrams in \(\CatStyle{C}\) which have to be small, since \(\CatStyle{C}\) is complete only relatively to small diagrams.
To ensure that, it is necessary to impose a further hypothesis, namely the solution set condition, which is set-theoretical in nature.
\item Preorders can be identified with small categories having at most one arrow in each homset, with meets and joins corresponding to limits and colimits in the associated category.
Remarkably, the existence of all meets in a preorder implies the existence of all joins and the other way around, as they can be defined in terms of each other.
This suggests that complete categories should be cocomplete, by defining colimits as limits of suitable diagrams.
Unfortunately, that only holds true for small categories, as the size of such diagrams can be as large as the size of the whole category.
\end{itemize}
Small categories are unaffected by size issues, but unfortunately, the only small complete categories are the complete preorders (the original source for this is, to the best of our knowledge, an unpublished theorem by Freyd).
This is the reason why completeness and cocompleteness are not generally regarded as equivalent notions: most complete categories are large, and thus potentially not cocomplete.

To overcome the aforementioned difficulties, we turn our attention to internal category theory, which is intrinsically unconcerned about size and features notable examples of complete categories, such as the category of modest sets in the effective topos (or, more precisely, in its subcategory of assemblies)  \autocite{hyland1982effective,hyland1988small,Hyland90discreteobjects}.
Our intuition is proven correct in that, in the internal context, the notion of completeness is particularly well-behaved: for example, complete internal categories are also cocomplete, and the adjoint functor theorem holds without any sort of solution set condition \autocite{ghiorzi2020complete}.

Standard category theory is intrinsically biased towards sets, in that (locally small) categories are implicitly enriched over \(\SetCat\), the category of sets and functions.
Unfortunately, in the internal context there is no immediate notion of large category: the usual substitute is via a theory of indexed categories, but that has the disadvantage that the notion of large is not intrinsic to the setting.
That limits the expressivity of the theory of internal categories, and indeed a considerable portion of standard category theory is problematic from the internal point of view: for example, the concept of presheaf cannot be defined and the Yoneda lemma cannot even be stated.

This paper aims to combine the pleasantness of internal categories with the expressivity of enriched categories by developing a theory of internal enrichment.
After recalling some convenient background material in \Cref{sec:int-cats,sec:mon-cats,sec:indexed-cats,sec:ext-int-cats}, in \Cref{sec:int-enriched-cats} we introduce the theory of internal enrichment over an internal monoidal category, obtained by internalizing the standard theory of enrichment over a monoidal category.
After that, we place internal enrichment in the rich landscape of notions of generalized enrichment.
It turns out, indeed, that internal enriched categories are closely related to both enriched indexed categories \autocite{shulman2008framed,shulman2013enriched} and enriched generalized multicategories \autocite{leinster1999generalized,Leinster02GeneralizedEnrichment,Leinster04HigherOperads}, as discussed in \Cref{sec:indexed-enriched-cats,sec:enriched-gen-multicats} respectively.
Thus, as a corollary, we establish a connection between the notions of enrichment for indexed categories and for generalized multicategories which, to the best of our knowledge, was previously unknown.

From an application perspective, we believe internal enrichment can be a valuable tool in the study of categorical models for polymorphism, in theoretical computer science.
Indeed, Eugenio Moggi originally suggested that the effective topos contains a small complete subcategory as a way to understand how realizability toposes give rise to models for impredicative polymorphism, and concrete versions of such models first appeared in \textcite{girard1972interpretation}.
Moreover, \textcite{reynolds1984polymorphism} noted that set theory is inadequate to treat polymorphism, and \textcite{pitts1987polymorphism} overcame the issue by internalizing the model in a suitable topos.
Finally, \textcite{hasegawa1994relational} has presented a model for higher-order polymorphic lambda-calculus based on enrichment over the category of partial equivalence relations, and noticed that inconveniently such category is incomplete, although it is internally complete in the effective topos \autocite{hyland1988small}.
Thus, internal and enriched categories are essential tools in the treatment of polymorphism. %
Internal enriched categories, being their common generalization and carrying the benefits of both, would remove the necessity of picking one and allow for a natural extension of the known models.

This paper is based on the author's doctoral research \autocite{ghiorzi2019thesis}.
Further developments on the theory of internal enrichment, particularly in the direction of a theory of completeness, are already contained in the dissertation.
We intend to present that material in a dedicated paper, whose preparation is in progress.

\section{Internal categories}\label{sec:int-cats}

In this section we quickly set some notation with regard to internal categories, without actually discussing their theory.
Such notation is mostly standard, as it is fundamentally similar to that used in well-known textbooks \autocite{MacLane98CategoriesWM,Borceaux94HandbookCA}.
A biref account of the theory of internal categories adopting the notation of this section can be found in \textcite{ghiorzi2020complete}.

In the context of this paper, let \(\E\) be a category with finite limits, which we regard as our ambient category.
To be precise, in particular we require \(\E\) to have a cartesian monoidal structure, that is, a monoidal structure given by a functorial choice of binary products \(\Prod \colon \E \Prod \E \to \E\) and a chosen terminal object \(\Terminal\).

We start by giving the definitions of internal category, functor and natural transformation of \(\E\).

\begin{definition}[internal category]
	An \Def{internal category} \(\A\) in \(\E\) is a diagram
	\begin{equation*}
		\begin{tikzcd}[column sep = large]
			A_0
				\ar[r, "{\Id[\A]}" description]
			& A_1
				\ar[l, "{\Source[\A]}"', shift right=2, bend right]
				\ar[l, "{\Target[\A]}", shift left=2, bend left]
			& A_1 \Pullback[\Source[\A]][\Target[\A]] A_1
				\ar[l, "{\Comp[\A]}"']
		\end{tikzcd}
	\end{equation*}
	in \(\E\) (where \(A_1 \Pullback[\Source[\A]][\Target[\A]] A_1\) is the pullback of \(\Source[\A]\) and \(\Target[\A]\)) satisfying the usual axioms for categories.
\end{definition}

As a category with finite limits, \(\E\) is a model for cartesian logic, or finite limit logic.
So, we will frequently use its internal language to ease the notation.
There are multiple accounts of the internal language of categories in the literature.
In particular, we shall follow \textcite{johnstone2002sketches,crole1993categories}, but, since we only make a basic use of the internal language, other references would be equally adequate.
We now introduce conventions to ease the use of the internal language in relation to internal categories, by bringing it closer to the standard notation of category theory.

\begin{notation}
	Given an internal category \(\A\) and, in some context, terms \(a_0 \colon A_0\), \(a_1 \colon A_0\) and \(f \colon A_1\), we shall write \(f \colon a_0 \to_{\A} a_1\)
	instead of (the conjunction of) the formulas \(\Source[\A](f) = a_0 \colon A_0\) and \(\Target[\A](f) = a_1 \colon A_0\).
	(The category to which the arrows belong can be omitted when it is already clear from the context.)
	Moreover, given terms
    \(g_0 \colon a_0 \to_{\A} a_{1}\) and
    \(g_1 \colon a_1 \to_{\A} a_{2}\),
    we shall write the composition of \(g_0\) and \(g_1\), i.e., the term \(g_1 \Comp[\A] g_0 \colon A_1\), as
    \[
    	a_0 \xrightarrow{g_0} a_1 \xrightarrow{g_1} a_2.
    \]
	Then, we can use the familiar notation for commuting diagrams even in the internal language.
\end{notation}

Leveraging the notation just introduced, we define functors of internal categories.
Then internal categories and their functors form a category \(\Cat[\E]\).

\begin{definition}[internal functor]
	Let \(\A\) and \(\B\) be internal categories in \(\E\).
	A \Def{functor} of internal categories \(F \colon \A \to \B\) is given by a pair of arrows \(F_0 \colon A_0 \to B_0\) and \(F_1 \colon A_1 \to B_1\) such that
	\begin{align*}
		f \colon a_0 \to_{\A} a_1 &\TiC F_1(f) \colon F_0(a_0) \to_{\B} F_0(a_1), \\
		a \colon A_0 &\TiC F_1 \Id[\A](a) = \Id[\B] F_0(a) \colon B_1,
	\end{align*}
	and satisfying the usual functoriality axioms.
	The composition \(G \Comp[\Cat[\E]] F \colon \A \to \C\) (shortened in \(GF\) for brevity) of two consecutive internal functors \(F \colon \A \to \B\) and \(G \colon \B \to \C\), and the identity functor \(\Id[\Cat[\E]](\A) \colon \A \to \A\) (shortened in \(\Id(\A)\) for brevity) for an internal category \(\A\), are defined in the usual way.
\end{definition}

The category of internal categories is well-behaved with respect to slicing, as the following remark makes clear.

\begin{remark}
	Let \(\E'\) be another category with finite limits, and \(F \colon \E \to \E'\) a functor preserving finite limits.
	Then, there is a functor \(F \colon \Cat[\E] \to \Cat[\E']\) (with abuse of notation) applying \(F\) to the underlying graph of internal categories.
\end{remark}

In the following remark, we notice some useful properties of \(\Cat[\E]\) in relation to slicing and change of base.

\begin{remark}
	\label{rmk:cat-reindexing}
	Let \(i \colon J \to I\) be an arrow in \(\E\).
	Then, there is an adjunction \(i_! \Adjoint i^* \colon \sfrac{\E}{I} \to \sfrac{\E}{J}\) where the functor \(i_! \colon \sfrac{\E}{J} \to \sfrac{\E}{I}\) is given by post-composition with \(i\), and the functor \(i^* \colon \sfrac{\E}{I} \to \sfrac{\E}{J}\) is given by pullback along \(i\).
	This adjunction extends to internal categories, yielding \(i_! \Adjoint i^* \colon \Cat[\sfrac{\E}{I}] \to \Cat[\sfrac{\E}{J}]\).
	In particular, the unique arrow \(!_I \colon I \to \Terminal\) yields an adjunction \(I_! \Adjoint I^* \colon \Cat[\E] \to \Cat[\sfrac{\E}{I}]\).
\end{remark}

We now define natural transformations of internal categories.
Then internal categories, together with their functors and the natural transformations between them, form a 2-category \(\Cat[\E]\) (denoted in the same way as its underlying 1-category with abuse of notation).

\begin{definition}[internal natural transformation]
	Let \(F, G \colon \A \to \B\) be functors of internal categories in \(\E\).	
	A \Def{natural transformation} of internal functors \(\alpha \colon F \to G \colon \A \to \B\) is given by an arrow \(\alpha \colon A_0 \to B_1\) such that
	\[
		a \colon A_0 \TiC \alpha(a) \colon F_0(a) \to_{\B} G_0(a)
	\]
	and satisfying the usual naturality axiom.
	Vertical and horizontal compositions of natural transformations and the identity natural transformation are defined in the usual way.
\end{definition}

To clarify a subtlety of the notation, notice that \(\Id[\A]\) denotes the identity of the category \(\A\), while \(\Id[\Cat[\E]](\A)\) denotes the identity functor on \(\A\) and \(\Id[\Cat[\E](\A, \B)](F)\) denotes the identity natural transformation on \(F \colon \A \to \B\).
When it is clear from the context, we might omit the subscript of \(\Id\) and write, for example, \(\Id(\A)\) in place of \(\Id[\Cat[\E]](\A)\) and \(\Id(F)\) in place of \(\Id[\Cat[\E](\A, \B)](F)\).

The following result is the internal version of the standard set-theoretic one, and it can be proved by a completely routine application of the internal language of \(\E\).

\begin{proposition}
	The category \(\Cat[\E]\) has finite limits induced point-wise by the corresponding limits in \(\E\).
	In particular, there is a terminal internal category \(\Terminal[\Cat[\E]]\) and a binary product \(\Prod[\Cat[\E]]\) of internal categories making \(\Cat[\E]\) a cartesian monoidal category.
\end{proposition}

It can also be useful to transport the structure of internal categories along a functor changing the base category.

\begin{proposition}\label{prop:int-base-change}
	Let \(F \colon \E \to \E'\) be a functor between categories with finite limits which preserves such finite limits.
	Then, there is an induced change-of-base functor \(F_{\bullet} \colon \Cat[\E] \to \Cat[\E']\), which also preserves finite limits.
\end{proposition}

There is also an obvious \Def{objects functor} \(U \colon \Cat[\E] \to \E\) sending an internal category to its underlying-object-of-objects.
Such functor preserves the cartesian monoidal structure.

We now mention a few remarkable examples of internal categories:
\begin{itemize}
	\item \(\Dis A\), the \Def{discrete category} over an object \(A\).
	\item \(\Ind A\), the \Def{indiscrete category} over an object \(A\).
	\item \(\A^{\Op}\), The opposite category of an internal category \(\A\).
\end{itemize}
Then, we have the following adjunctions:
\begin{itemize}
	\item \(\Dis \Adjoint U \Adjoint \Ind\).
	\item \({(\MathDash)}^{\Op} \Adjoint {(\MathDash)}^{\Op}\).
\end{itemize}
An alternative, more abstract way to look at the discrete category with respect to \Cref{rmk:cat-reindexing} is to notice that \(\Dis A\) is (equivalent to) \(A_! A^* \Terminal[\Cat[\E]]\).

Sometimes, and especially when using the internal language, the notation denoting the object or morphism component of a functor can be cumbersome.
Thus, in the following sections we shall adopt a common convention in category theory and omit to make such component explicit when it is clear from the context which one we are referring to.
For example, given an internal functor \(F \colon \A \to \B\), in context \(a \colon A_0, f \colon A_1\) we shall write \(F(a)\) for \(F_0(a)\) and \(F(f)\) for \(F_1(f)\).

\section{Internal monoidal categories}\label{sec:mon-cats}

We could not conceivably present a notion of enrichment without a suitable notion of monoidal category to enrich over.
We introduce here the definitions, in the internal language of \(\E\), of the notions of monoidal category, functor and natural transformation.

\begin{definition}[internal monoidal category]
	An \Def{internal monoidal category} is an internal category \(\IntV\) in \(\E\) equipped with functors
	\begin{description}
		\item[Monoidal product] \(\MonProd[\IntV] \colon \IntV \Prod[\Cat[\E]] \IntV \to \IntV\), and
		\item[Monoidal unit] \(\MonUnit[\IntV] \colon \Terminal[\Cat[\E]] \to \IntV\),
	\end{description}
	and natural isomorphisms
	\begin{description}
		\item[Associator] \(\MonAssoc[\IntV] \colon (\pi_1 \MonProd[\IntV] \pi_2) \MonProd[\IntV] \pi_3 \to \pi_1 \MonProd[\IntV] (\pi_2 \MonProd[\IntV] \pi_3) \colon \IntV \Prod[\Cat[\E]] \IntV \Prod[\Cat[\E]] \IntV \to \IntV\),
		\item[Left unitor] \(\MonUnitorL[\IntV] \colon \MonUnit[\IntV] \MonProd[\IntV] \Id(\IntV) \to \Id(\IntV) \colon \IntV \to \IntV\), and
		\item[Right unitor] \(\MonUnitorR[\IntV] \colon \Id(\IntV) \MonProd[\IntV] \MonUnit[\IntV] \to \Id(\IntV) \colon \IntV \to \IntV\),
	\end{description}
	such that, in context \(a, b, c, d \colon V_0\), the axioms
	\begin{gather}
		\tag{Triangle}
		\begin{tikzcd}[ampersand replacement = \&]
			( a \MonProd[\IntV][0] \MonUnit[\IntV] ) \MonProd[\IntV][0] b
					\ar[rr, "{\MonAssoc[\IntV](a, \MonUnit[\IntV], b)}"]
					\ar[dr, "{\MonUnitorR[\IntV](a) \MonProd[\IntV][1] \Id[\IntV](b)}"']
				\& \& a \MonProd[\IntV][0] ( \MonUnit[\IntV] \MonProd[\IntV][0] b )
					\ar[dl, "{\Id[\IntV](a) \MonProd[\IntV][1] \MonUnitorL[\IntV](b)}"] \\
			\& a \MonProd[\IntV][0] b
		\end{tikzcd} \\
		\tag{Pentagon}
		\begin{tikzcd}[ampersand replacement = \&, column sep = 0]
			\& ( ( a \MonProd[\IntV][0] b ) \MonProd[\IntV][0] c ) \MonProd[\IntV][0] d
				\ar[dl, "{\MonAssoc[\IntV](a, b, c) \MonProd[\IntV][1] \Id[\IntV](d)}"']
				\ar[dr, "{\MonAssoc[\IntV](a \MonProd[\IntV][0] b, c, d)}"] \\
			( a \MonProd[\IntV][0] ( b \MonProd[\IntV][0] c ) ) \MonProd[\IntV][0] d
				\ar[d, "{\MonAssoc[\IntV](a, b \MonProd[\IntV][0] c, d)}"']
			\&\& ( a \MonProd[\IntV][0]  b ) \MonProd[\IntV][0] ( c \MonProd[\IntV][0] d )
				\ar[d, "{\MonAssoc[\IntV](a, b, c \MonProd[\IntV][0] d)}"] \\
			a \MonProd[\IntV][0] ( ( b \MonProd[\IntV][0] c ) \MonProd[\IntV][0] d )
				\ar[rr, "{\Id[\IntV](a) \MonProd[\IntV][1] \MonAssoc[\IntV](b, c, d)}"']
			\&\& a \MonProd[\IntV][0] ( b \MonProd[\IntV][0] ( c \MonProd[\IntV][0] d ) )
		\end{tikzcd}
	\end{gather}
	hold.
\end{definition}

The previous definition is a direct internalization of the standard definition of monoidal category, and that alone should suffice to persuade us of its correctness.
If we were still skeptical, though, it could also be argued that, since small monoidal categories are pseudomonoids in the 2-category of categories, then internal monoidal categories in \(\E\) must be pseudomonoids in the 2-category \(\Cat[\E]\), which is what our definition amounts to.

We then proceed with the definition of monoidal functor.
	
\begin{definition}[internal monoidal functor]
	An \Def{internal monoidal functor} \((F, \epsilon, \mu) \colon \IntV \to \IntW\) is given by an internal functor \(F \colon \IntV \to \IntW\) and coherence natural isomorphisms
	\begin{align*}
		\epsilon &\colon \MonUnit[\IntW] \to F \MonUnit[\IntV] \colon \MonUnit[\Cat[\E]] \to \IntW \\
			\shortintertext{and}
		\mu &\colon F \MonProd[\IntW] F \to F ( \MathDash \MonProd[\IntV] \MathDash) \colon \IntV \Prod[\Cat[\E]] \IntV \to \IntW
	\end{align*}
	such that, in context \(a, b, c \colon V_0\), the axioms
	\begin{gather*}
		\begin{tikzcd}[ampersand replacement = \&, column sep = large]
			\big( F( a ) \MonProd[\IntW] F( b ) \big) \MonProd[\IntW] F( c )
					\ar[r, bend left = 10, "{\MonAssoc[\IntW]\big( F( a ), F( b ), F( c ) \big)}"]
					\ar[d, "{\mu(a, b) \MonProd[\IntW][1] \Id[\IntW] F(c)}"']
				\& F( a ) \MonProd[\IntW] \big( F( b ) \MonProd[\IntW] F( c ) \big)
					\ar[d, "{\Id[\IntW] F(a) \MonProd[\IntV][1] \mu(b, c)}"] \\
			F( a \MonProd[\IntV] b ) \MonProd[\IntW] F( c )
					\ar[d, "{\mu( a \MonProd[\IntV][0] b, c)}"']
				\& F( a ) \MonProd[\IntW] F( b \MonProd[\IntV] c )
					\ar[d, "{\mu( a , b \MonProd[\IntV][0] c)}"] \\
			F\big( ( a \MonProd[\IntV] b ) \MonProd[\IntV] c \big)
					\ar[r, "{F \big( \MonAssoc[\IntV](a, b, c) \big)}"']
				\& F \big( a \MonProd[\IntV] ( b \MonProd[\IntV] c ) \big)
		\end{tikzcd} \\
		\begin{tikzcd}[ampersand replacement = \&]
			\MonUnit[\IntW] \MonProd[\IntW] F( a )
					\ar[rrr, "{\MonUnitorL[\IntW] \big( F( a ) \big)}"]
					\ar[dr, "{\epsilon \MonProd[\IntW] F( a )}"']
				\& \& \& F ( a ) \\
			\& F( \MonUnit[\IntV] ) \MonProd[\IntW] F( a )
					\ar[r, "{\mu(\MonUnit[\IntV], a)}"']
				\& F( \MonUnit[\IntV] \MonProd[\IntV] a )
					\ar[ur, "{F \MonUnitorL[\IntV](a)}"']
		\end{tikzcd} \\
		\begin{tikzcd}[ampersand replacement = \&]
			F( a ) \MonProd[\IntW] \MonUnit[\IntW]
					\ar[rrr, "{\MonUnitorR[\IntW] \big( F( a ) \big)}"]
					\ar[dr, "{F( a ) \MonProd[\IntW] \epsilon}"']
				\& \& \& F ( a ) \\
			\& F( a ) \MonProd[\IntW] F( \MonUnit[\IntV] )
					\ar[r, "{\mu(a, \MonUnit[\IntV])}"']
				\& F_0( a \MonProd[\IntV] \MonUnit[\IntV] )
					\ar[ur, "{F \MonUnitorR[\IntV](a)}"']
		\end{tikzcd}
	\end{gather*}
	hold.
\end{definition}

Then, we define natural trasformations of monoidal functors.
	
\begin{definition}[internal monoidal natural transformation]
	An \Def{internal monoidal natural transformation} \(\alpha \colon (F, \epsilon_F, \mu_F) \to (G, \epsilon_G, \mu_G) \colon \IntV \to \IntW\) is a natural transformation \(\alpha \colon F \to G \colon \IntV \to \IntW\) such that, in context \(a, b, \colon V_0\), the axioms
	\begin{gather*}
		\begin{tikzcd}[ampersand replacement = \&, column sep = large]
			F(a) \MonProd[\IntW] F(b)
					\ar[r, "{\alpha(a) \MonProd[\IntW] \alpha(b)}"]
					\ar[d, "{\mu_F(a, b)}"']
				\& G(a) \MonProd[\IntW] G(b)
					\ar[d, "{\mu_G(a, b)}"] \\
			F(a \MonProd[\IntV] b)
					\ar[r, "{\alpha(a \MonProd[\IntV] b)}"']
				\& G(a \MonProd[\IntV] b)
		\end{tikzcd} \\
		\begin{tikzcd}[ampersand replacement = \&, column sep = small]
			F(\MonUnit[\IntV])
				\ar[rr, "{\alpha(\MonUnit[\IntV])}"]
			\&\& G(\MonUnit[\IntV]) \\
			\& \MonUnit[\IntW]
				\ar[ul, "{\epsilon_F}"]
				\ar[ur, "{\epsilon_G}"']
		\end{tikzcd}
	\end{gather*}
	hold.
\end{definition}

It is routine to check in the internal language that the data above gives 2-categories, so we can give the following definitions.

\begin{definition}[category of internal monoidal categories]
	Internal monoidal categories and monoidal functors (and monoidal transformations) in \(\E\) form a (2-)category \(\MonCat[\E]\).
\end{definition}

As for plain internal categories (see \Cref{prop:int-base-change}), we can transport the monoidal structure along a functor changing the base category.

\begin{proposition}\label{prop:int-mon-base-change}
	Let \(F \colon \E \to \E'\) be a functor between categories with finite limits which preserves such finite limits.
	Then, there is an induced change-of-base functor \(F_{\bullet} \colon \MonCat[\E] \to \MonCat[\E']\), which also preserves finite limits.
\end{proposition}

Finally, notice that there is an underlying-internal-category (2-)functor
\[
	U_{\Cat[\E]} \colon \MonCat[\E] \to \Cat[\E]
\]
sending monoidal categories and functors (and natural transformations) to, respectively, their underlying internal categories and functors (and natural transformations).

\section{Indexed categories}\label{sec:indexed-cats}

Indexed categories, while not playing a direct role in the definition of internal enrichment, will be essential to their understanding in relation to other notions of enrichment.

Indexed categories have been treated extensively in the literature, and the main ideas are long established.
However, we shall refer to the recent exposition given in \textcite{shulman2008framed,shulman2013enriched}, since these sources are also needed in regard to the notion of enriched indexed category.

To begin with, we shall state the definition of indexed category.

\begin{definition}[indexed category]
	An \Def{\(\E\)-indexed category} is a pseudofunctor \(\E^{\Op} \to \Cat\), where \(\Cat\) is the 2-category of categories, functors, and natural transformations.
\end{definition}

Consider the following notable example, which will turn out to be useful later on.

\begin{example}
	\label{exm:self-indexing}
	The self-indexing of \(\E\) is the \(\E\)-indexed category whose fiber over an object \(X\) is the slice category \(\sfrac{\E}{X}\) and where the reindexing along \(f \colon X \to Y\) is given by pullback along \(f\).
\end{example}

There is a strict relation between the theory of indexed categories and that of fibration, as established by the following, classic result.

\begin{theorem}
	\label{thm:groth-constr}
	An \(\E\)-indexed category \(\IdxC \colon \E^{\Op} \to \Cat\) is, via the \emph{Grothendieck construction}, equivalent to a cloven fibration \(\int \IdxC \to \E\).
\end{theorem}

Now we want to extend the previous ideas to the context of monoidal categories.
We begin by giving the notion of indexed monoidal categories.

\begin{definition}[indexed monoidal category]	
	An \Def{\(\E\)-indexed monoidal category} is a pseudofunctor \(\IdxW \colon \E^{\Op} \to \MonCat\), where \(\MonCat\) is the 2-category of monoidal categories, strong monoidal functors, and monoidal transformations.
\end{definition}

A suitable notion of monoidal fibration is required to establish a relation with indexed monoidal categories, so we recall that in the following definition. 

\begin{definition}[monoidal fibration]
	Let \(\ExtV\) be a monoidal category.
	A monoidal fibration is a cloven fibration \(\ExtV \to \E\) such that the underlying functor is strict monoidal (with \(\E\) regarded as cartesian monoidal) and the tensor product in \(\ExtV\) preserves the choice of cartesian arrows.
\end{definition}

For a general monoidal base category the notions of indexed monoidal category and of monoidal fibration do not correspond under the Grothendieck construction.
Indeed, if \(\IdxW \colon \E'^{\Op} \to\MonCat\) is an \(\E'\)-indexed monoidal category, then, in the cloven fibration \(\int \IdxW \to \E'\), it is evident that \(\int \IdxW\) has tensor products only for objects in the same fiber, and the result is still an object in that fiber.
On the other hand, if \(\ExtV \to \E'\) is a monoidal fibration, and \(A\) and \(B\) are objects of \(\ExtV\) lying over the objects \(X\) and \(Y\) of \(\E'\) respectively, then the tensor product \(A \MonProd[\ExtV] B\) lies over \(X \MonProd[\E'] Y\).
However, it is folklore that, in case the monodial structure on \(\E'\) is given by the product, i.e., \(\E'\) is cartesian monoidal, such as our ambient category \(\E\) is, then there is a correspondence.

We now introduce some convenient notation for use in the next result.

\begin{notation}
	Let \(F \colon \C \to \E\) be a cloven fibration, and \(f \colon X \to Y\) an arrow in \(\E\).
	Then, call \(f^* \colon F^{-1}(Y) \to F^{-1}(X)\) the lifting functor from the fiber along \(F\) over \(Y\) to the fiber over \(X\).
\end{notation}

We then state and give a sketch proof of the theorem analogous to \Cref{thm:groth-constr}, in the context of monoidal categories.

\begin{theorem}\label{thm:Grothendieck-mon}
	An \(\E\)-indexed monoidal category \(\IdxW\) is, via the \emph{Grothendieck construction}, equivalent to a monoidal fibration \(\int \IdxW \to \E\).
\end{theorem}
\begin{proof}[Proof (sketch)]
	Let the pseudo-functor \(\IdxW \colon \E^{\Op} \to \MonCat\) be an index monoidal category and \(\int~\IdxW~\to~\E\) the fibration yielded by the Grothendieck construction.
	Then, \(\int \IdxW\) has a monoidal structure.
	Indeed, \(\MonUnit[\int \IdxW] = \MonUnit[\IdxW(\Terminal)]\).
	Let \(X\) and \(Y\) be objects of \(\E\).
	Let \(A\) be an object of \(\IdxW(X)\) and \(B\) one of \(\IdxW(Y)\).
	Then,
	\[
		A \MonProd[\int \IdxW] B = {(\pi_{Y})}^{*} A \MonProd[\IdxW(X \Prod Y)] {(\pi_{X})}^{*} B
	\]
	where \(\pi_{Y} \colon X \Prod Y \to X\) and \(\pi_{X} \colon X \Prod Y \to Y\).
	With this monoidal structure on \(\int \IdxW\), the fibration \(\int \IdxW \to \E\) is strict monoidal.

	Conversely, let \(P \colon \ExtV \to \E\) be a strict monoidal fibration and \(\IdxW \colon \E^{\Op} \to \Cat\) the pseudofunctor defined by the fibers of \(P\).
	Then, \(\IdxW\) is an \(\E\)-indexed monoidal category, that is, it restricts to \(\IdxW \colon \E^{\Op} \to \MonCat\).
	Indeed, let \(X\) be an object of \(\E\), and \(!_X \colon X \to \Terminal\) the unique such arrow.
	Then \(\MonUnit[\IdxW(X)] = {(!_{X})}^{*}\MonUnit[\ExtV]\).
	Let \(A\) and \(B\) be a objects of \(\IdxW(X)\).
	Then
	\[
		A \MonProd[\IdxW(X)] B = \Delta^{*}(A \MonProd[\ExtV] B)
	\]
	where \(\Delta \colon X \to X \Prod X\) is the diagonal arrow.
\end{proof}

Notice that the proof makes essential use of the assumption that \(\E\) has (at least) finite products.

\section{Externalization of internal categories}\label{sec:ext-int-cats}

The last piece of background material that we present concerns the relationship between internal and indexed categories, and makes an essential use of the theory of indexed categories from \Cref{sec:indexed-cats}.
For the central notion of externalization of an internal category, we shall follow the exposition of \textcite{hyland1988small,Hyland90discreteobjects}.

Let \(\A\) be a category in \(\E\) and \(X\) an object of \(\E\).
We regard an arrow \(X \to A_0\) as representing an indexed family of objects of \(\A\) over the indexing object \(X\).
Given two such indexed families \(x_0, x_1 \colon X \to A_0\), consider the pullback
\[
	\begin{tikzcd}
		{(x_0, x_1)}^{*}A_1
			\ar[d, "{p}"'] \ar[r] \ar[dr, phantom, "{\lrcorner}" very near start]
		& A_1
			\ar[d, "{(\Source[\A], \Target[\A])}"] \\
		X
			\ar[r, "{(x_0, x_1)}"']
		& A_0 \Prod A_0.
	\end{tikzcd}
\]
Then the sections of \(p\) represent indexed families of arrows of \(\A\) with domain \(x_0\) and codomain \(x_1\).
Given another family \(x_2 \colon X \to A_0\), the composition in \(\A\) restricts to an indexed composition
\[
	\Comp[\A]_{|x_0, x_1, x_2} \colon {(x_1, x_2)}^{*}A_1 \Prod {(x_0, x_1)}^{*}A_1 \to {(x_0, x_2)}^{*}A_1
\]
inducing a composition of indexed families of arrows: given two families of arrows \(s_0 \colon X \to {(x_0, x_1)}^{*}A_1\) and \(s_1 \colon X \to {(x_1, x_2)}^{*}A_1\), their composition is defined as
\[
	s_1 \Comp[\Externalization{\A}[X]] s_0 \DefEq X \xrightarrow{(s_1, s_0)} {(x_1, x_2)}^{*}A_1 \Prod {(x_0, x_1)}^{*}A_1 \xrightarrow{\Comp[\A]_{|x_0, x_1, x_2}} {(x_0, x_2)}^{*}A_1.
\]
Moreover, a family of objects \(x \colon X \to A_0\) induces a family of identity arrows \(\Id[\Externalization{\A}[X]](x) \colon X \to {(x, x)}^{*}A_1\).
These data form the category \(\Externalization{\A}[X]\) of indexed families of objects and morphisms of \(\A\) over \(X\).

Given a reindexing \(u \colon X' \to X\), precomposition reindexes a family of objects \(x \colon X \to A_0\) over \(X\) the family \(xu\) over \(X'\);	a family of arrows \(s \colon X \to {(x_0, x_1)}^{*}A_1\) is reindexed to \(u^{*}s \colon X' \to {(u x_0, u x_1)}^{*}A_1\) by pulling back the section \(s\) along \((x_0, x_1)\).
That gives a functor \(u^{*} \colon \Externalization{\A}[X] \to \Externalization{\A}[X']\).

The above discussion leads to the following result.

\begin{proposition}
	For \(\A\) an internal category in \(\E\), there is an indexed category \(\Externalization{\A}\) given by \(\Externalization{\A}(X) \DefEq \Externalization{\A}[X]\) and \(\Externalization{\A}(u) \DefEq u^{*}\).
\end{proposition}

\begin{remark}
	Notice that the indexed category arising from an internal one is given by a strict (rather than merely a pseudo) functor \(\E^{\Op} \to \Cat\).
	Then, evidently, internal categories yield rather special indexed categories, and not all indexed categories can be obtained from an internal one.
\end{remark}

Moreover, the construction extends to the monoidal context, as shown in the following proposition.

\begin{proposition}\label{prop:index-mon-cat}
	Let \(\IntV\) be an internal monoidal category in \(\E\).
	Then, \(\Externalization{\IntV}\) is an indexed monoidal category on \(\E\).
\end{proposition}
\begin{proof}
	Let \(X\) be an object in \(\E\).
	Then, \(\Externalization{\IntV}[X]\) has a monoidal structure induced by that of \(\IntV\).
	The monoidal product on objects is defined as
	\[
		(X \xrightarrow{x} V_0) \MonProd[\Externalization{\IntV}[X]] (X \xrightarrow{x'} V_0) \DefEq X \xrightarrow{(x, x')} V_0 \Prod V_0 \xrightarrow{\MonProd[\IntV]} V_0.
	\]
	To define the monoidal product of arrows, let
	\begin{gather*}
		(X \xrightarrow{x_0} V_0) \xrightarrow{f \colon X \to (x_0, x_1)^* V_1} (X \xrightarrow{x_1} V_0) \\
		\shortintertext{and}
		 (X \xrightarrow{x_0'} V_0) \xrightarrow{f' \colon X \to (x_0', x_1')^* V_1} (X \xrightarrow{x_1'} V_0)
	\end{gather*}
	be arrows of \(\Externalization{\IntV}[X]\), and notice that \(\MonProd[\IntV]\) restricts to
	\[
		(x_0, x_1)^* V_1 \Prod (x_0', x_1')^* V_1 \to (x_0 \MonProd[\Externalization{\IntV}[X]] x_0', x_1 \MonProd[\Externalization{\IntV}[X]] x_1')^*V_1.
	\]
	Then the monoidal product of arrows \(f \MonProd[\Externalization{\IntV}[X]] f'\) is given by the arrow
	\[
		X \xrightarrow{(f, f')} (x_0, x_1)^* V_1 \Prod (x_0', x_1')^* V_1 \xrightarrow{\MonProd[\IntV]} (x_0 \MonProd[\IntV] x_0', x_1 \MonProd[\IntV] x_1')^*V_1.
	\]
	The monoidal unit \(\MonUnit[\Externalization{\IntV}[X]]\) is defined by the constant family indexed by \(X\) on the monoidal unit of \(\IntV\).
	The structural isomorphisms, associator and unitors are defined point-wise.
	Moreover, reindexing preserves the monoidal product.
\end{proof}

\begin{remark}
	The strictness of the monoidal products of the fibers of the indexed monoidal category \(\Externalization{\IntV}\) obtained from an internal monoidal category \(\IntV\) is the same as that of the original monoidal product of \(\IntV\), so it will generally not be strict monoidal.
	Still, the reindexing functors for \(\Externalization{\IntV}\) strictly preserve the monoidal structure, regardless of how strict the monoidal product of \(\IntV\) is.
	That means that the (actually strict) functor \(\Externalization{\IntV} \colon \E \to \MonCat\) factorizes through the 2-category of (non-necessarily-strict) monoidal categories, strict monoidal functors and monoidal natural transformations.
	Such a category is quite uncommon, since normally there is little use for strict monoidal functors, especially between non-strict monoidal categories.
	Nonetheless, this shows that the indexed monoidal categories arising from internal monoidal categories are rather special ones.
\end{remark}

\begin{remark}
	The fiber \(\Externalization{\A}[X]\) over an object \(X\) is enriched over \(\sfrac{\E}{X}\):
	\begin{description}
		\item[Homset:] \(\Hom[\Externalization{\A}[X]](x_0, x_1) \DefEq {(x_0, x_1)}^{*}A_1 \xrightarrow{p} X\).
		\item[Composition:] \(\Comp[\Externalization{\A}[X]](x_0, x_1, x_2) \DefEq \Comp[\A]_{|x_0, x_1, x_2}\).
		\item[Identity:] \(\Id[\Externalization{\A}[X]](x)\).
	\end{description}
	Reindexing is compatible with this structure, in that the reindexing of the homset is the same as the homset of the reindexing.
	More explicitly, given a reindexing \(u \colon X' \to X\), by pullback-pasting we have that
	\[
		u^{*}{(x_0, x_1)}^{*}A_1 \Iso {(x_0 u, x_1 u)}^{*}A_1.
	\]
	In fact, the reindexing functor \(u^* \colon \Externalization{\A}[X] \to \Externalization{\A}[X']\) is a fully-faithful functor of enriched categories.
	Then \(\Externalization{\A}\) is an indexed enriched category over the self-indexing of \(\E\) (see \Cref{exm:self-indexing}).
	Equivalently, it is the locally internal category over \(\E\) whose underlying indexed category is (up to natural isomorphism) \(\Externalization{\A}\) as an indexed category \autocite{shulman2013enriched,johnstone2002sketches}.
\end{remark}

As stated in \Cref{thm:groth-constr}, indexed categories are equivalent to cloven fibrations.
So, we can give an abstract definition of the externalization of an internal category as follows.

\begin{definition}[externalization]
	The \Def{externalization} of an internal category \(\A\) is the total category for the fibration associated to the indexed category \(\Externalization{\A}\).
	With abuse of notation, we denote the externalization of \(\A\) with \(\Externalization{\A}\), and context will usually suffice to distinguish between the use of the notation as a fibration or as an indexed category.
\end{definition}

For practical purposes, it is useful to make the previous definition more explicit.
The externalization of \(\A\) is the category given by the data
\begin{description}
	\item[Objects:] families of objects of \(\A\) indexed over objects of \(\E\), that is, arrows \(X \to A_0\) with \(X\) in \(\E\).
	\item[Morphisms:] an arrow \((x \colon X \to A_0) \to (y \colon Y \to A_0)\) is given by a reindexing \(u \colon X \to Y\) and a family of arrows \(x \to yu\), that is, a section of the projection \(p \colon {(x, yu)}^{*}A_1 \to X\).
	\item[Composition:] the composition is given by
	\begin{multline*}
		(X \xrightarrow{x} A_0)
			\xrightarrow{\big( X \xrightarrow{u} Y, X \xrightarrow{f} {(x, yu)}^{*}A_1 \big)}
		(Y \xrightarrow{y} A_0)
			\xrightarrow{\big( Y \xrightarrow{v} Z, Y \xrightarrow{g} {(y, zv)}^{*}A_1 \big)}
		(Z \xrightarrow{z} A_0) \\
		\DefEq (X \xrightarrow{x} A_0)
			\xrightarrow{\big( X \xrightarrow{u} Y \xrightarrow{v} Z, u^{*} g \Comp[\Externalization{\A}[X]] f \colon X \to {(x, zvu)}^{*}A_1 \big)}
		(Z \xrightarrow{z} A_0).
	\end{multline*}
	\item[Identity:] the family of identity arrows.
\end{description}

Let \(F \colon \A \to \B\) be a functor of internal categories.
Then, there is a functor of fibered categories \(\Externalization{F} \colon \Externalization{\A} \to \Externalization{\B}\) defined on objects as
\[
	\Externalization{F}(X \xrightarrow{x} A_0) \DefEq X \xrightarrow{x} A_0 \xrightarrow{F_0} B_0
\]
and on morphisms as
\begin{multline*}
	\Externalization{F} \Big(
		\big( X \xrightarrow{x} A_0 \big)
			\xrightarrow{
				\big(X \xrightarrow{u} Y, X \xrightarrow{f} {(x, yu)}^{*}A_1 \big)
			}
		\big( Y \xrightarrow{y} A_0 \big)
	\Big) \\
	\DefEq \big( X \xrightarrow{x} A_0 \xrightarrow{F_0} B_0 \big)
		\xrightarrow{
			\Big(
				X \xrightarrow{u} Y,
				X \xrightarrow{f} {(x, yu)}^{*}A_1 \xrightarrow{F_1} {(F_0 x, F_0 yu )}^{*}B_1
			\Big)
		}
	\big( Y \xrightarrow{y} A_0 \xrightarrow{F_0} B_0 \big)
\end{multline*}
which restricts to a functor on the fibers \(\Externalization{F}[X] \colon \Externalization{\A}[X] \to \Externalization{\B}[X]\).

Let \(\alpha \colon F \to G \colon \A \to \B\) be a natural transformation.
Then, there is a natural transformation of fibered categories \(\Externalization{\alpha} \colon \Externalization{F} \to \Externalization{G} \colon \Externalization{\A} \to \Externalization{\B}\), defined as
\[
	\Externalization{\alpha}(X \xrightarrow{x} A_0) \DefEq
		\big( X \xrightarrow{x} A_0 \xrightarrow{F_0} B_0 \big)
			\xrightarrow{( \Id[X], X \xrightarrow{x} A_0 \xrightarrow{\alpha} (F_0 x, G_0 x)^* B_1 )}
		\big( X \xrightarrow{x} A_0 \xrightarrow{G_0} B_0 \big)
\]
which restricts to a natural transformation on the fibers \(\Externalization{\alpha}[X] \colon \Externalization{F}[X] \to \Externalization{G}[X]\).

\begin{remark}
	If \(\IntV\) is a monoidal category in \(\E\), then \(\Externalization{\IntV}\) is an indexed monoidal category by \Cref{prop:index-mon-cat}.
	By \Cref{thm:Grothendieck-mon}, it follows that the externalization \(\Externalization{\IntV}\) has an induced monoidal structure.
	Explicitly, the monoidal product on objects is given by
	\[
		(X \xrightarrow{x} V_0) \MonProd[\Externalization{\IntV}] (Y \xrightarrow{y} V_0) \DefEq X \Prod Y \xrightarrow{x \Prod y} V_0 \Prod V_0 \xrightarrow{\MonProd[\IntV]} V_0.
	\]
	The monoidal product on arrows
	\begin{gather*}
		(X \xrightarrow{x} V_0)
			\xrightarrow{\big( X \xrightarrow{u} Y, X \xrightarrow{f} {(x, yu)}^{*}V_1 \big)}
		(Y \xrightarrow{y} V_0) \\
	\shortintertext{and}
		(Z \xrightarrow{z} V_0)
			\xrightarrow{\big( Z \xrightarrow{v} W, Z \xrightarrow{g} {(z, wv)}^{*}V_1 \big)}
		(W \xrightarrow{w} V_0)	
	\end{gather*}
	is the arrow \((X \xrightarrow{x} V_0) \MonProd[\Externalization{\IntV}] (Z \xrightarrow{z} V_0) \to (Y \xrightarrow{y} V_0) \MonProd[\Externalization{\IntV}] (W \xrightarrow{w} V_0)\)
	indexed by \(u \Prod v \colon X \Prod Z \to Y \Prod W\) and given by
	\[
		f \MonProd[\IntV] g \colon X \Prod Z \to {(x \MonProd[\IntV] z, yu \MonProd[\IntV] wv)}^{*}V_1.
	\]
	The monoidal unit is \(\MonUnit[\IntV] \colon \Terminal[\E] \to V_0\).
	Finally, the structural isomorphisms are induced by those of \(\IntV\).
\end{remark}

\section{Internal enriched categories}\label{sec:int-enriched-cats}

We are finally ready to introduce the main topic of this paper: the theory internal enrichment.
We shall derive the necessary notions by the process of internalization of the theory of standard enrichment.
Substantially, that amounts to translating the definitions from enriched category theory into the internal language of the ambient category.
In other words, we take advantage of the fact that the axioms of the theory of enrichment are expressible in cartesian logic.

From now on, let \(\IntV\) be an internal monoidal category in \(\E\).
We then define the notion of enrichment in \(\IntV\) internal to the ambient category \(\E\).

\begin{definition}[internal enriched category]
	An \Def{internal \(\IntV\)-enriched category} \(\X\) in \(\E\), or \(\IntV\)-category, is given by the following data.
	\begin{description}
		\item[Underlying object:] an object \(X\) of \(\E\).
		\item[Internal hom:] a morphism \(\Hom[\X] \colon X \times X \to V_0\).
		\item[Composition:] a morphism \({\Comp[\X]} \colon X \times X \times X \to V_1\) such that, in context \(x_0, x_1, x_2 \colon X\),
		\[
			\Comp[\X](x_0, x_1, x_2) \colon \Hom[\X](x_1, x_2) \MonProd[\IntV] \Hom[\X](x_0, x_1) \to_{\IntV} \Hom[\X](x_0, x_2).
		\]
		\item[Identity:] a morphism \(\Id[\X] \colon X \to V_1\) such that, in context \(x \colon X\),
		\[
			\Id[\X] \colon \MonUnit[\IntV] \to_{\IntV} \Hom[\X](x, x).
		\]
	\end{description}
	Moreover, it has to satisfy the following axioms (in context \(x_0, x_1, x_2, x_3 \colon X\)).
	\begin{gather}
		\label[axiom]{eq:int_enr_cat_axiom_associativity}
		\hspace{-2em}\begin{tikzcd}[ampersand replacement = \&, column sep = small, row sep = large]
			\big( \Hom[\X](x_2, x_3) \MonProd[\IntV] \Hom[\X](x_1, x_2) \big) \MonProd[\IntV] \Hom[\X](x_0, x_1)
					\ar[d, "{\MonAssoc[\IntV] \big( \Hom[\X](x_2, x_3), \Hom[\X](x_1, x_2), \Hom[\X](x_0, x_1) \big)}" description]
					\ar[r, bend left = 10, "{{\Comp[\X]}(x_1, x_2, x_3) \MonProd[\IntV] \Hom[\X](x_0, x_1)}"]
				\& \Hom[\X](x_1, x_3) \MonProd[\IntV] \Hom[\X](x_0, x_1)
					\ar[dd, "{{\Comp[\X]}(x_0, x_1, x_3)}"] \\
			\Hom[\X](x_2, x_3) \MonProd[\IntV] \big( \Hom[\X](x_1, x_2) \MonProd[\IntV] \Hom[\X](x_0, x_1) \big)
					\ar[d, "{\Hom[\X](x_2, x_3) \MonProd[\IntV] {\Comp[\X]}(x_0, x_1, x_2)}"'] \\
			\Hom[\X](x_2, x_3) \MonProd[\IntV] \Hom[\X](x_0, x_2)
					\ar[r, "{{\Comp[\X]}(x_0, x_2, x_3)}"']
				\& \Hom[\X](x_0, x_3)
		\end{tikzcd} \\
		\label[axiom]{eq:int_enr_cat_axiom_left_unit_law}
		\begin{tikzcd}[ampersand replacement = \&, column sep = small, row sep = large]
			\MonUnit[\IntV] \MonProd[\IntV] \Hom[\X](x_0, x_1)
					\ar[rr, "{\MonUnitorL[\IntV] \big( \Hom[\X](x_0, x_1) \big) }"]
					\ar[dr, "{\Id[\X](x_1) \MonProd[\IntV] \Hom[\X](x_0, x_1) }"']
				\&\& \Hom[\X](x_0, x_1) \\
			\& \Hom[\X](x_1, x_1) \MonProd[\IntV] \Hom[\X](x_0, x_1)
					\ar[ur, "{{\Comp[\X]}(x_0, x_1, x_1)}"']
		\end{tikzcd} \\
		\label[axiom]{eq:int_enr_cat_axiom_right_unit_law}
		\begin{tikzcd}[ampersand replacement = \&, column sep = small, row sep = large]
			\Hom[\X](x_0, x_1) \MonProd[\IntV] \MonUnit[\IntV]
					\ar[rr, "{\MonUnitorR[\IntV] \big( \Hom[\X](x_0, x_1) \big) }"]
					\ar[dr, "{ \Hom[\X](x_0, x_1) \MonProd[\IntV] \Id[\X](x_0) }"']
				\&\& \Hom[\X](x_0, x_1) \\
			\& \Hom[\X](x_0, x_0) \MonProd[\IntV] \Hom[\X](x_0, x_1)
					\ar[ur, "{{\Comp[\X]}(x_0, x_0, x_1)}"']
		\end{tikzcd}
	\end{gather}
\end{definition}

Notice how the conventions on the internal language of \(\E\) allow one to express those axioms in a form very close to that used to define standard enriched categories.

\begin{example}
	Let \(\ExtV\) be a small monoidal category.
	Then, \(\ExtV\) is an internal category in \(\SetCat\), and internal \(\ExtV\)-enriched categories in \(\SetCat\) are standard (small) \(\ExtV\)-enriched categories.
\end{example}

Continuing in the style of the previous definition, we give a notion of internal enriched functor, by translating the standard definition into the internal language.

\begin{definition}[internal enriched functor]
	Let \(\X\) and \(\Y\) be \(\IntV\)-enriched categories.
	A \Def{\(\IntV\)-enriched functor}, or \(\IntV\)-functor, \(F \colon \X \to \Y\) is given by the following data.
	\begin{description}
		\item[Objects component:] an arrow \(F_0 \colon X \to Y\).
		\item[Morphisms component:] an arrow \(F_1 \colon X \times X \to V_1\) such that, in context \(x_0, x_1 \colon X \),
			\[
				F_1(x_0, x_1) \colon \Hom[\X](x_0, x_1) \to_{\IntV} \Hom[\Y] \big( F_0(x_0), F_0(x_1) \big).
			\]
	\end{description}
	Moreover, it has to satisfy the following axioms (in context \(x_0, x_1, x_2 \colon X\)).
	\begin{gather}
		\begin{tikzcd}[ampersand replacement = \&, column sep = small, row sep = large]
		\label[axiom]{eq:int_enr_cat_axiom_functoriality}
			\Hom[\X](x_1, x_2) \MonProd[\IntV] \Hom[\X](x_0, x_1)
					\ar[d, "{F_1(x_1, x_2) \MonProd[\IntV][1] F_1(x_0, x_1)}"']
					\ar[r, "{{\Comp[\X]}(x_0, x_1, x_2)}"]
				\& \Hom[\X](x_0, x_2)
					\ar[d, "{F_1(x_0, x_2)}"] \\
			\Hom[\Y] \big( F_0(x_1), F_0(x_2) \big) \MonProd[\IntV] \Hom[\Y] \big( F_0(x_0), F_0(x_1) \big)
					\ar[r, bend right = 15, "{{\Comp[\Y]} \big( F_0(x_0), F_0(x_1), F_0(x_2) \big)}"']
				\& \Hom[\Y] \big( F_0(x_0), F_0(x_2) \big)
		\end{tikzcd} \\
		F_1 \Id[\X] = \Id[\Y] F_0
		\label[axiom]{eq:int_enr_cat_axiom_functoriality-id}
	\end{gather}
\end{definition}

We would expect the definition above to provide a category \(\Cat[\E][\IntV]\) of internal \(\IntV\)-enriched categories and functors.
We present the relevant data for that.

The composition of \(\IntV\)-functors \(F \colon \X \to \Y\) and  \(G \colon \Y \to \Z\) is defined, in context \(x_0, x_1 \colon X\), as follows.
\begin{align*}
	{(G F)}_0 &\DefEq G_0 F_0 \colon X \to Z \\
	{(G F)}_1(x_0, x_1) &\DefEq
	\begin{tikzcd}[ampersand replacement = \&]
		\Hom[\X](x_0, x_1)
			\ar[d, "{F_1(x_0, x_1)}"] \\
		\Hom[\Y] \big( F_0(x_0), F_0(x_1) \big)
			\ar[d, "{G_1 \big( F_0(x_0), F_0(x_1) \big)}"] \\
		\Hom[\Z] \big( G_0 F_0 (x_0), G_0 F_0(x_1) \big)
	\end{tikzcd}
\end{align*}

The identity \(\IntV\)-functor \(\Id[\Cat[\E][\IntV]](\X) \colon \X \to \X\) on \(\X\) is defined as follows.
\begin{align*}
	{(\Id[\Cat[\E][\IntV]](\X))}_0 &\DefEq X \xrightarrow{\Id(X)} X \\
	{(\Id[\Cat[\E][\IntV]](\X))}_1 &\DefEq X \Prod X \xrightarrow{\Hom[\X]} V_0 \xrightarrow{\Id[\IntV]} V_1
\end{align*}

It is just an exercise in the internal language to prove that the data so defined yield a category, as stated in the following proposition.

\begin{proposition}
	Composition and identity of internal enriched functors strictly satisfy associativity and unitarity.
	Then, \(\IntV\)-enriched categories and functors form a category \(\Cat[\E][\IntV]\).
\end{proposition}

\begin{example}
	There is an underlying-object functor \(U \colon \Cat[\E][\IntV] \to \E\) sending \(\IntV\)-enriched categories to their underlying object, and \(\IntV\)-enriched functors to their object-component.
\end{example}

\begin{example}
	Let \(X\) be an object of \(\E\).
	The indiscrete \(\IntV\)-enriched category \(\Ind(X)\) on \(X\) is given by
	\[
		\Hom[\Ind(X)] \DefEq X \Prod X \xrightarrow{!} \Terminal[\E] \xrightarrow{\MonUnit[\IntV]} V_0.
	\]
	The rest of the structure follows from that.
	Analogously, a morphism \(f \colon X \to Y\) induces an indiscrete \(\IntV\)-enriched functor \(\Ind(f) \colon \Ind(X) \to \Ind(Y)\).
	Then, there is a functor \(\Ind \colon \E \to \Cat[\E][\IntV]\).
\end{example}

\begin{remark}
To define the discrete \(\IntV\)-enriched category over an object of \(\E\), we would need to assume some extra hypothesis.
Firstly, we would need to be able to tell whether two elements of the underlying object of the \(\IntV\)-enriched category are equal.
Secondly, we would need an initial object in \(\IntV\) to be the homset of non-equal elements of the underlying object.
Both hypothesis do not hold in general.
For example, the first one does not hold in the effective topos.
\end{remark}

Finally, again by translating the standard definition into the internal language, we give the definition of internal enriched natural transformation.

\begin{definition}[internal enriched natural transformation]
	Let \(\X\) and \(\Y\) be \(\IntV\)-enriched categories, and \(F\) and \(G\) be \(\IntV\)-enriched functors \(\X \to \Y\).
	A \Def{\(\IntV\)-enriched natural transformation}, or \(\IntV\)-natural transformation, \(\alpha \colon F \to G \colon \X \to \Y\) is given by an arrow \(\alpha \colon X \to V_1\) such that, in context \(x \colon X\),
	\[
		\alpha(x) \colon \MonUnit[\IntV] \to_{\IntV} \Hom[\Y] \big( F_0(x), G_0(x) \big)
	\]	
	and satisfying, in context \(x_0, x_1 \colon X\), the following axiom.
	\begin{equation}
		\label[axiom]{eq:int_enr_cat_axiom_naturality}
		\begin{tikzcd}[ampersand replacement = \&, column sep = small, row sep = large]
			\MonUnit[\IntV] \MonProd[\IntV] \Hom[\X](x_0, x_1)
					\ar[r, bend left = 15, "{\alpha(x_1) \MonProd[\IntV] F_1(x_0, x_1)}"]
				\& \Hom[\Y] \big( F_0(x_1), G_0(x_1) \big) \MonProd[\IntV] \Hom[\Y] \big( F_0(x_0), F_0(x_1) \big)
					\ar[d, "{\Comp[\Y]\big( F_0(x_0), F_0(x_1), G_0(x_1) \big)}"] \\
			\Hom[\X](x_0, x_1)
					\ar[u, "{\MonUnitorL[\IntV]^{-1} \big( \Hom[\X](x_0, x_1) \big)}" description]
					\ar[d, "{\MonUnitorR[\IntV]^{-1} \big( \Hom[\X](x_0, x_1) \big)}" description]
				\& \Hom[\Y] \big( F_0(x_0), G_0(x_1) \big) \\
			\Hom[\X](x_0, x_1) \MonProd[\IntV] \MonUnit[\IntV]
					\ar[r, bend right = 15, "{G_1(x_0, x_1) \MonProd[\IntV] \alpha(x_0)}"']
				\& \Hom[\Y] \big( G_0(x_0), G_0(x_1) \big) \MonProd[\IntV] \Hom[\Y] \big( F_0(x_0), G_0(x_0) \big)
					\ar[u, "{\Comp[\Y]\big( F_0(x_0), G_0(x_0), G_0(x_1) \big)}"']
		\end{tikzcd}
	\end{equation}
\end{definition}

We would expect the definition above to provide a 2-category of internal \(\IntV\)-enriched categories, functors and natural transformations.
We present the relevant data for that.

Consider \(\IntV\)-categories, \(\IntV\)-functors and \(\IntV\)-natural transformations as shown in the diagram:
\[
	\begin{tikzcd}[column sep = large]
		\W
			\ar[r, "L"]
		& \X
			\ar[rr, shift left, bend left = 45, "F", ""{name=F, below}]
			\ar[rr, "G" description, ""{name=Ga, above}, ""{name=Gb, below}]
			\ar[rr, shift right, bend right = 45, "H"', ""{name=H, above}]
		&& \Y
			\arrow[Rightarrow, from=F, to=Ga, "\alpha"]
			\arrow[Rightarrow, from=Gb, to=H, "\beta"]
			\ar[r, "{R}"]
		& \Z
	\end{tikzcd}
\]
Vertical composition of \(\IntV\)-natural transformations \(\beta \Comp \alpha \colon F \to H \colon X \to Y\) is defined, in context \(x \colon X\) by
\[
	(\beta \Comp \alpha)(x) \DefEq
		\begin{tikzcd}
			\MonUnit[\IntV]
				\ar[d, "{\MonUnitorL[\IntV]^{-1}(\MonUnit[\IntV]) = \MonUnitorR[\IntV]^{-1}(\MonUnit[\IntV])}"] \\
			\MonUnit[\IntV] \MonProd[\IntV] \MonUnit[\IntV]
				\ar[d, "{\beta(x) \MonProd[\IntV] \alpha(x) }"] \\
			\Hom[\X] \big( G_0(x), H_0(x) \big) \MonProd[\IntV] \Hom[\X] \big( F_0(x), G_0(x) \big)
				\ar[d, "{\Comp[\Y] \big( F_0(x), G_0(x), H_0(x) \big)}"] \\
			\Hom[\X] \big(F_0(x), H_0(x) \big)
		\end{tikzcd}
\]
The left whiskering \(\alpha \HComp L \colon F L \to G L \colon \W \to \Y\) is defined, in context \(w \colon W\), as \((\alpha \HComp L)(w) \DefEq \alpha L_0(w)\).
The right whiskering \(R \HComp \beta \colon R \Comp G \to R \Comp H \colon \X \to \Z\) is defined, in context \(x \colon X\), as follows.
\[
	(R \HComp \beta)(x) \DefEq
		\begin{tikzcd}
			\MonUnit[\IntV]
				\ar[d, "{\beta(x)}"] \\
			\Hom[\Y] \big( G_0(x), H_0(x) \big)
				\ar[d, "{R_1 \big( G_0(x), H_0(x) \big)}"] \\
			\Hom[\Z] \big( R_0 G_0(x), R_0 H_0(x) \big)
		\end{tikzcd}
\]
The identity \(\IntV\)-natural transformation \(\Id[F] \colon F \to F \colon \X \to \Y\) is defined as \(\Id[F] \DefEq \Id[\Y] F_0\).

It is just an exercise in the internal language to prove that vertical composition and identity of internal enriched natural transformation strictly satisfy associativity and unitarity.
Thus, a pair of \(\IntV\)-enriched categories yield a category of functors and natural transformations, as stated in the following proposition.

\begin{proposition}
	Given \(\IntV\)-enriched categories \(\X\) and \(\Y\), the \(\IntV\)-enriched functors \(\X \to \Y\) and natural transformations between them form a category \(\Cat[\E][\IntV](\X, \Y)\).
\end{proposition}

Moreover, horizontal and vertical composition of \(\IntV\)-enriched natural transformations strictly satisfy the interchange laws, thus yielding an enrichment in \(\Cat\).
Equivalently, internal enriched categories, functors and natural transformations form a 2-category, as stated in the following result, whose proof is again an exercise in the internal language.

\begin{proposition}
	\(\IntV\)-enriched categories, functors, and natural transformations form a strict 2-category \(\Cat[\E][\IntV]\).
\end{proposition}

By abuse of notation, we call \(\Cat[\E][\IntV]\) both the category of \(\IntV\)-enriched categories and their functors, and the 2-category of \(\IntV\)-enriched categories, their functors and their natural transformations.
As a consequence, given two \(\IntV\)-enriched categories \(\X\) and \(\Y\), we will denote by \(\Cat[\E][\IntV](\X, \Y)\) both the hom-set of \(\IntV\)-enriched functors \(\X \to \Y\) and the hom-category of \(\IntV\)-enriched functors \(\X \to \Y\) and their natural transformations.
Context will usually suffice to determine in which sense the notation is being used.

\begin{remark}
	Let \(\X\) be a \(\IntV\)-enriched category.
	There is an underlying \(\E\)-category \(U(\X)\), such that \({U(\X)}_0 \DefEq X\) and \({U(\X)}_1\) is the subobject of \(X \Prod X \Prod V_1\) given, in context \((x_0, x_1, f) \colon {U(\X)}_1\), by
	\[
		f \colon \MonUnit[\IntV] \to_{\IntV} \Hom[\X](x_0, x_1)
	\]
	with the first and second projections as source and target.
	The composition is defined, in context \((x_1, x_2, g), (x_0, x_1, f) \colon {U(\X)}_1 \Pullback[\Source][\Target] {U(\X)}_1\), as follows.
	\begin{equation*}
			(x_1, x_2, g) \Comp[U(\X)] (x_0, x_1, f) \DefEq \big( x_0, x_2, \Comp[\X](x_0, x_1, x_2) \Comp[\IntV] (g \MonProd[\IntV] f) \big)
	\end{equation*}
	
	Let \(F \colon \X \to \Y\) be a \(\IntV\)-enriched functor.
	There is an underlying functor \(U(F) \colon U(\X) \to U(\Y)\) in \(\E\), with \({U(F)}_0\) defined as \(F_0\) and \({U(F)}_1 (x_0, x_1, f)\), in context \((x_0, x_1, f) \colon {U(\X)}_1\), as the tuple
	\begin{equation*}
			\left( F_0(x_0), F_0(x_1), \MonUnit[\IntV] \xrightarrow{f} \Hom[\X](x_0, x_1) \xrightarrow{F_1(x_0, x_1)} \Hom[\X] \big( F_0(x_0), F_0(x_1) \big) \right)
	\end{equation*}
	in \({U(\Y)}_1\).
	
	Let \(\alpha \colon F \to G \colon \X \to \Y\) be a \(\IntV\)-enriched natural transformation.
	There is an underlying natural transformation \(U(\alpha) \colon U(F) \to U(G) \colon U(\X) \to U(\Y)\) in \(\E\), defined, in context \(x \colon {U(\X)}_0\) as \(U(\alpha)(x) \DefEq \big( F_0(x), G_0(x), \alpha(x) \big)\).

	Those data yield the underlying-category-in-\(\E\) 2-functor \(U \colon \Cat[\E][\IntV] \to \Cat[\E]\).
\end{remark}

We now consider the issue of the change of base.
In this context, though, there are two sensible such notions, one coming from internal category theory and one from enriched category theory.
Indeed, we can change both the ambient category and the enriching category.

To begin, let's state the internal version of the standard result, changing the enriching category.

\begin{proposition}
	Let \(\IntV'\) be another monoidal category in \(\E\), and \(F \colon \IntV \to \IntV'\) a monoidal functor.
	Then there is an induced 2-functor \(F_\bullet \colon \Cat[\E][\IntV] \to \Cat[\E][\IntV']\).
\end{proposition}
\begin{proof}
	Let \(\X\) be a \(\IntV\)-category.
	Define a \(\IntV'\)-category \(F_\bullet(\X)\) on \(X\) given by the following data.
	\begin{description}
		\item[Internal hom:] \(\Hom[F_\bullet(\X)] \DefEq X \Prod X \xrightarrow{\Hom[\X]} V_0 \xrightarrow{F_0} V'_0\).
		\item[Composition:] \(\Comp[F_\bullet(\X)] \DefEq X \Prod X \Prod X \xrightarrow{\Comp[\X]} V_1 \xrightarrow{F_1} V'_1\).
		\item[Identity:] \(\Id[F_\bullet(\X)] \DefEq X \xrightarrow{\Hom[\X]} V_1 \xrightarrow{F_1} V'_1\).
	\end{description}
	
	Let \(G \colon \X \to \Y\) be a \(\IntV\)-functor.
	Define a \(\IntV'\)-functor \(F_\bullet(G) \colon F_\bullet(\X) \to F_\bullet(\Y)\), with the same object component as \(G\) and arrow component given by \({(F_\bullet(G))}_1 \DefEq F_1 G_1\).

	Let \(\alpha \colon G \to G' \colon \X \to \Y\) be a \(\IntV\)-natural transformation.
	Define a \(\IntV'\)-natural transformation \(F_\bullet(\alpha) \colon F_\bullet(G) \to F_\bullet(G') \colon F_\bullet(\X) \to F_\bullet(\Y)\) as \(F_\bullet(\alpha) \DefEq F_1 \alpha\).

	The axioms for the above definitions hold because of the functoriality of \(F\).
\end{proof}

Finally, let's check that changing the ambient category induces a 2-functorial operation on internal enriched categories, just as it does on internal categories (see \Cref{prop:int-base-change}).

\begin{proposition}
	Let \(\E'\) be another finitely complete category and \(F \colon \E \to \E'\) a functor preserving finite limits.
	By \Cref{prop:int-mon-base-change}, there is an induced monoidal category \(F_\bullet (\IntV)\) in \(\E'\).
	Then \(F\) induces a 2-functor \(F_\bullet \colon \Cat[\E][\IntV] \to \Cat[\E'][F_{\bullet}(\IntV)]\).
\end{proposition}
\begin{proof}
	Let \(\X\) be a \(\IntV\)-category.
	Define a \(F_\bullet (\IntV)\)-category \(F_\bullet(\X)\) on \(F(X)\) by applying the functor \(F\) to the structural arrows \(\Hom[\X]\), \(\Comp[\X]\) and \(\Id[\X]\) of \(\X\).
	That gives a \(F_\bullet (\IntV)\)-enriched category because \(F\) preserves finite-limit logic, in terms of which internal enriched categories are defined.
	Analogously, define \(F_\bullet\) on \(\IntV\)-enriched functors and natural transformations.
\end{proof}

In the remaining \Cref{sec:indexed-enriched-cats,sec:enriched-gen-multicats} of the paper we shall perform some sanity checks on the new notion of internal enrichment by comparing it to other analogous notions from the literature.
By doing so, we hope to persuade the reader that the theory we propose is sound and that it deserves a place among the other generalized notions of enrichment.

\section{Indexed enriched categories}\label{sec:indexed-enriched-cats}

In \Cref{sec:indexed-cats} we recalled the notion of indexed monoidal category.
Unsurprisingly, there is a notion of enrichment over such a category \autocite{shulman2013enriched} which is a fibrational generalization of the standard enrichment.
This comes in two versions: a general indexed version and a version which Shulman calls ``small''.
The latter is in a sense a hybrid notion, having an internal as well as an indexed aspect.
We shall then compare both of them to internal enrichment, and find that they are closely related.

First, we give an outline of the notions of small \(\IdxW\)-category, of functors between such categories, and of natural transformations between such functors.
For brevity we will omit some diagrammatic axioms, referring to \textcite{shulman2013enriched} for those.
In this section, let \(\IdxW\) be an \(\E\)-indexed monoidal category.
Moreover, if \(f \colon B \to A\) is a morphism in \(\E\) and \(H\) is an object in \(\IdxW(A)\), we shall write \(H(f)\) as a convenient notation for the object \(\IdxW(f)(H)\) of \(\IdxW(B)\).

\begin{definition}[small \(\IdxW\)-category]
	A \Def{small \(\IdxW\)-category} \(\A\) consists of the following data:
	\begin{itemize}
		\item an object \(A\) of \(\E\).
		\item an object \(\Hom[\A]\) of \(\IdxW(A \Prod A)\).
		\item a morphism \(\Id[\A] \colon \MonUnit[\IdxW(A)] \to \Hom[\A](\Delta)\) where \(\Delta \colon A \to A \Prod A\) is the diagonal.
		\item A morphism of \(\IdxW(A \Prod A \Prod A)\)
			\[
				\Comp[\A] \colon \Hom[\A](\pi_2, \pi_3) \MonProd[\IdxW(A \Prod A \Prod A)] \Hom[\A](\pi_1, \pi_2) \to \Hom[\A](\pi_1, \pi_2)
			\]
			where \(\pi_1, \pi_2, \pi_3 \colon A \Prod A \Prod A \to A\) are projections.
	\end{itemize}
	Moreover, it has to satisfy the associativity and unitarity axioms \autocite{shulman2013enriched}.
\end{definition}
	
\begin{definition}[functor of small \(\IdxW\)-categories]
	A \Def{functor of small \(\IdxW\)-categories} \(F \colon \A \to \B\) consists of the following data:
	\begin{itemize}
		\item a morphism \(F_0 \colon A \to B\) of \(\E\).
		\item a morphism \(F_1 \colon \Hom[\A] \to \Hom[\B](F_0, F_0)\) of \(\IdxW(A \Prod A)\).
	\end{itemize}
	Moreover, it has to satisfy the functoriality axioms \autocite{shulman2013enriched}.
\end{definition}
	
\begin{definition}[natural transformation of small \(\IdxW\)-categories]
	A \Def{natural transformation of small \(\IdxW\)-categories} \(\alpha \colon F \to G \colon \A \to \B\) consists of a morphism
	\[
		\alpha \colon \MonUnit[\IdxW(A)] \to \Hom[\B] \big( (F_0, G_0) \Delta \big)
	\]
	satisfying the naturality axiom \autocite{shulman2013enriched}.
\end{definition}

We shall denote with \(\ICat[\E][\IdxW]\) the (2-)category of small \(\IdxW\)-categories and their functors (and the natural transformation between those).

Recall from \Cref{sec:ext-int-cats} that the externalization of an internal monoidal category \(\IntV\) is a monoidal indexed category \(\Externalization{\IntV}\) over \(\E\).
Thus, we can investigate the relationship that subsists between \(\IntV\)-enriched categories and small \(\Externalization{\IntV}\)-categories,
and it turns out that they are the same thing in a very strict sense: their definitions coincide!

\begin{proposition}
	To give a \(\IntV\)-enriched category (functor, natural transformation) is to give a small \(\Externalization{\IntV}\)-category (functor, natural transformation).
	Thus, the categories \(\Cat[\E][\IntV]\) and \(\ICat[\E][\Externalization{\IntV}]\) are isomorphic.
\end{proposition}
\begin{proof}
	A small \(\Externalization{\IntV}\)-category \(\X\) is given by the following data:
	\begin{itemize}
		\item an object \(X\) of \(\E\).
		\item an object \(\Hom[\X] \colon X \times X \to V_0\) of \(\Externalization{\IntV}[X \Prod X]\).
		\item a morphism of \(\Externalization{\IntV}\)
			\[
				(X \xrightarrow{!_X} \Terminal[\E] \xrightarrow{\MonUnit[\IntV]} V_0)
					\xrightarrow{\big( X \xrightarrow{\Delta_X} X \Prod X, X \xrightarrow{\Id[\X]} ( \MonUnit[\IntV] !_X , \Hom[\X] \Delta_X )^{*}V_1 \big)}
				(X \Prod X \xrightarrow{\Hom[\X]} V_0).
			\]
		\item a morphism of \(\Externalization{\IntV}\)
			\[
				(X \Prod X \Prod X \xrightarrow{\Hom[\X](\pi_2, \pi_3) \MonProd[\IntV] \Hom[\X](\pi_1, \pi_2)} V_0)
					\to
				(X \Prod X \xrightarrow{\Hom[\X]} V_0)
			\]
			over \(X \Prod X \Prod X \xrightarrow{(\pi_1, \pi_3)} X \Prod X\), given by
			\[
				X \Prod X \Prod X \xrightarrow{\Comp[\X]} {\big( \Hom[\X](\pi_2, \pi_3) \MonProd[\IntV] \Hom[\X](\pi_1, \pi_2), \Hom[\X](\pi_1, \pi_3) \big)}^{*}V_1.
			\]
	\end{itemize}
	Moreover, such data have to satisfy associativity and unitarity axioms.
	But these are precisely the same data that yield an internal \(\IntV\)-enriched category.
	
	Analogously, to give a functor or a natural transformation of small \(\Externalization{\IntV}\)-categories is to give a functor or a natural transformation of internal \(\IntV\)-enriched categories.
\end{proof}

Now we present the notion of indexed category enriched in an indexed monoidal category \autocite{shulman2013enriched}.
For that, we shall extend a notation that we have consistently used in the internal context to standard enriched categories: if \(F \colon \ExtV \to \ExtV'\) is a lax monoidal functor and \(\ExtA\) is a \(\ExtV\)-enriched category, then \(F_{\bullet}(\ExtA)\) is the induced \(\ExtV'\)-enriched category.

\begin{definition}[indexed \(\IdxW\)-category]
	An \Def{indexed \(\IdxW\)-category} \(\B\) consists of the following data:
	\begin{itemize}
		\item for each \(X\) object of \(\E\), a \(\IdxW(X)\)-category \(\B^{X}\).
		\item for each \(f \colon X \to Y\) in \(\E\), a fully faithful \(\IdxW(X)\)-functor \(f^{*} \colon {(f^*)}_{\bullet}( \B^{Y} )~\to~\B^{X}\).
		\item for each \(f \colon X \to Y\) and \(g \colon Y \to Z\) in \(\E\), a \(\IdxW(X)\)-natural isomorphism \({( g f )}^{*} \Iso f \Comp {(f^{*})}_{\bullet}(g)\) (where we implicitly identify \({(f^{*})}_{\bullet} {(g^{*})}_{\bullet} \B^{Z}\) with \({(gf^{*})}_{\bullet} \B^{Z}\) in the domains of these functors).
		\item for each \(X\) object of \(\E\), a \(\IdxW(X)\)-natural isomorphism \({(\Id[X])}^{*} \Iso \Id[\B^{X}]\).
	\end{itemize}
	Moreover, for every \(f \colon X \to Y\), \(g \colon Y \to Z\) and \(h \colon Z \to K\) in \(\E\), it has to satisfy the axioms for associativity and unitarity, analogous to those for ordinary indexed categories, by making the following diagrams of isomorphisms commute.
	\begin{gather*}
		\begin{tikzcd}[ampersand replacement = \&]
			{(hgf)}^*
					\ar[d] \ar[r]
				\&f^* \Comp {(f^*)}_\bullet \big( {(hg)}^* \big)
					\ar[d] \\
			{(gf)}^* \Comp {\big( {(gf)}^* \big)}_\bullet (h^*)
					\ar[d]
				\&f^* \Comp (f^*)_\bullet \big( g^* \Comp {(g^*)}_\bullet(h^*) \big)
					\ar[d] \\
			f^* \Comp {(f^*)}_\bullet (g^*) \Comp {\big( {(gf)}^* \big)}_\bullet (h^*)
					\ar[r]
				\&f^* \Comp {(f^*)}_\bullet (g^*) \Comp {(f^*)}_\bullet {(g^*)}_\bullet (h^*)
		\end{tikzcd} \\
		\begin{tikzcd}[column sep = 0, ampersand replacement = \&]
			\&{(\Id[\X])}^* \Comp {\big({(\Id[\X])}^* \big)}_\bullet (f^*)
					\ar[dr] \\
			{(f \Id[\X])}^*
					\ar[ur] \ar[rr, equal]
			\&\&f^*
		\end{tikzcd}
		\begin{tikzcd}[column sep = 0, ampersand replacement = \&]
			\&f^* \Comp {(f^*)}_\bullet \big( {(\Id[\Y])}^* \big)
					\ar[dr] \\
			{(\Id[\Y] f)}^*
					\ar[ur]
					\ar[rr, equal]
			\&\&f^*
		\end{tikzcd}
	\end{gather*}
\end{definition}
		
\begin{definition}[functor of indexed \(\IdxW\)-categories]
	An \Def{indexed \(\IdxW\)-functor} \(\mathcal{F} \colon \B \to \B'\) consists, for every object \(X\) of \(\E\), of a \(\IdxW(X)\)-enriched functor \(\mathcal{F}^{X} \colon \B^{X} \to {\B'}^{X}\) together with, for every \(f \colon X \to Y\), an isomorphism \(\mathcal{F}^{X} \Comp f^{*} \Iso f^{*} \Comp {( f^{*} )}_{\bullet} ( \mathcal{F}^{Y} )\).
	Such data have to satisfy the functoriality axioms by making the following diagrams of isomorphisms commute, for every \(f \colon X \to Y\) and \(g \colon Y \to Z\) in \(\E\).
	\begin{gather*}
		\begin{tikzcd}[ampersand replacement = \&]
			\mathcal{F}^I \Comp (gf)^*
					\ar[r] \ar[d]
				\&\mathcal{F}^X \Comp f^* \Comp (f^*)_\bullet(g^*)
					\ar[d] \\
			(gf)^{*} \Comp {\big( (gf)^{*} \big)}_{\bullet} ( \mathcal{F}^{Z} )
					\ar[d]
				\&f^{*} \Comp {( f^{*} )}_{\bullet} ( \mathcal{F}^{Y} ) \Comp (f^*)_\bullet(g^*)
					\ar[d] \\
			f^* \Comp (f^*)_\bullet (g^*) \Comp (f^*)_\bullet \big( (g^*)_\bullet (\mathcal{F}^Z) \big)
					\ar[dr, bend right=10]
				\&f^{*} \Comp {( f^{*} )}_{\bullet} ( \mathcal{F}^{Y} \Comp g^* )
					\ar[d] \\
			\&f^{*} \Comp {( f^{*} )}_{\bullet} \big( g^{*} \Comp {( g^{*} )}_{\bullet} (\mathcal{F}^Z) \big)
		\end{tikzcd} \\
		\begin{tikzcd}[ampersand replacement = \&]
			\mathcal{F}^X \Comp (\Id[\X])^*
					\ar[r] \ar[dr]
				\&(\Id[\X])^* \Comp ((\Id[\X])^*)_\bullet (\mathcal{F}^X)
					\ar[d] \\
			\&\mathcal{F}^X
		\end{tikzcd}
	\end{gather*}
\end{definition}

\begin{definition}[natural transformation of indexed \(\IdxW\)-categories]
	An indexed \(\IdxW\)-natural transformation \(\alpha \colon \mathcal{F} \to \mathcal{G} \colon \B \to \B'\) consists, for every object \(X\) of \(\E\), of a \(\IdxW(X)\)-natural transformation \(\alpha^X \colon \mathcal{F}^X \to \mathcal{G}^X \colon \B^X \to {\B'}^X\), satisfying naturality axioms by making the following diagram commute, for every \(f \colon X \to Y\).
	\[
		\begin{tikzcd}
			\mathcal{F}^{X} \Comp f^{*}
					\ar[d] \ar[r]
				&f^{*} \Comp {( f^{*} )}_{\bullet} ( \mathcal{F}^{Y} )
					\ar[d] \\
			\mathcal{F}^{X} \Comp f^{*}
					\ar[r]
				&f^{*} \Comp {( f^{*} )}_{\bullet} ( \mathcal{F}^{Y} )
		\end{tikzcd}
	\]
\end{definition}

With the data thus defined (plus the obvious notions of compositions and identities) we can define a 2-category of indexed enriched categories.

\begin{definition}[category of indexed \(\IdxW\)-categories]
	We denote with \(\ICat[\E][\IdxW]\) the 2-category of indexed \(\IdxW\)-categories, their functors and the natural transformations between them.
\end{definition}

By abuse of notation, we shall denote with \(\ICat[\E][\IdxW]\) also the mere 1-category of indexed \(\IdxW\)-categories and their functors.
Usually, the context is sufficient to distinguish when the notation is being used referring to the 1-category or the 2-category.

The relationship between \(\IntV\)-enriched categories and indexed \(\Externalization{\IntV}\)-categories is more complicated.
We will prove that \(\IntV\)-enriched categories are a sub-case of indexed \(\Externalization{\IntV}\)-categories, in the sense precisely stated in \Cref{prop:V-cat-index-cat,prop:V-cat-index-cat-fullsubcat}.

First, let \(\X\) be a \(\IntV\)-enriched category and let us define an indexed \(\Externalization{\IntV}\)-category \(\Externalization{\X}\).
Given an indexing object \(I\) of \(\E\), define the \(\Externalization{\IntV}[I]\)-enriched category \(\Externalization{\X}[I]\) as follows:
\begin{description}
	\item[Objects:] \(I\)-indexed families \(x \colon I \to X\) of elements of \(X\).
	\item[Internal hom:] \(\Hom[\Externalization{\X}[I]](x_0 \colon I \to X, x_1 \colon I \to X) \DefEq I \xrightarrow{(x_0, x_1)} X \Prod X \xrightarrow{\Hom[\X]} V_0\).
	\item[Composition:] \(\Comp[\Externalization{\X}[I]](x_0, x_1, x_2) \DefEq I \xrightarrow{(x_0, x_1, x_1)} X \Prod X \Prod X \xrightarrow{\Comp[\X]} V_1\).
	\item[Identity:] \(\Id[\Externalization{\X}[I]](x) \DefEq I \xrightarrow{x} X \xrightarrow{\Id[\X]} V_1\).
\end{description}
Let \(f \colon I \to J\) be a re-indexing.
Define the \(\Externalization{\IntV}[I]\)-functor \(f^* \colon {(f^*)}_{\bullet} (\Externalization{\X}[J]) \to \Externalization{\X}[I]\) as follows.
\begin{align*}
	f^*_0 ( J \xrightarrow{x} X )
		&\DefEq I \xrightarrow{f} J \xrightarrow{x} X \\
	f^*_1 ( J \xrightarrow{x_0} X, J \xrightarrow{x_1} X )
		&\DefEq \Hom[{(f^*)}_{\bullet} (\Externalization{\X}[J])](x_0, x_1) \xrightarrow{\Id} \Hom[\Externalization{\X}[I]]( x_0 f, x_1 f )
\end{align*}
Since \(f^*_1(x_0, x_1)\) is the identity of \(\Hom[\X](x_0 f, x_1 f)\) as an object of \(\Externalization{\IntV}[I]\), then \(f^*\) is full and faithful, as required by the definition.
The rest of the structure is given by canonical isomorphisms verifying the axioms.

Secondly, let \(F \colon \X \to \Y\) be a \(\IntV\)-enriched functor and let us define an indexed \(\Externalization{\IntV}\)-enriched functor \(\Externalization{F} \colon \Externalization{\X} \to \Externalization{\Y}\) induced by \(F\).
For an indexing object \(I\), define the \(\Externalization{\IntV}[I]\)-enriched functor \(\Externalization{F}[I] \colon \Externalization{\X}[I] \to \Externalization{\Y}[I]\) as follows:
\begin{description}
	\item[Objects component:] \(\Externalization{F}[I](I \xrightarrow{x} X) \DefEq I \xrightarrow{x} X \xrightarrow{F_0} Y\).
	\item[Morphisms component:] \(\Externalization{F}[I](I \xrightarrow{x_0} X, I \xrightarrow{x_1} X) \DefEq I \xrightarrow{(x_0, x_1)} X \Prod X \xrightarrow{F_1} V_1\).
\end{description}
Notice that, for any reindexing \(f \colon I \to J\), we have an equality \(\Externalization{F}[I] \Comp f^* = f^* \Comp (f^*)_\bullet (\Externalization{F}[J])\), meaning that the axioms for indexed \(\Externalization{\IntV}\)-enriched functors are automatically satisfied.

Finally, let \(\alpha \colon F \to G \colon \X \to \Y\) be a \(\IntV\)-enriched natural transformation and let us define an indexed \(\Externalization{\IntV}\)-natural transformation \(\Externalization{\alpha} \colon \Externalization{F} \to \Externalization{G} \colon \Externalization{\X} \to \Externalization{\Y}\) induced by \(\alpha\).
Let \(I\) be an indexing object and define the \(\Externalization{\IntV}[I]\)-enriched natural transformation \(\Externalization{\alpha}[I] \colon \Externalization{F}[I] \to \Externalization{G}[I] \colon \Externalization{\X}[I] \to \Externalization{\Y}[I]\) as \(\Externalization{\alpha}[I](x \colon I \to X) \DefEq \alpha x\).
The naturality condition for indexed \(\Externalization{\IntV}\)-natural transformations is trivially satisfied because the defining isomorphisms of indexed \(\Externalization{\IntV}\)-functors \(\Externalization{F}\) and \(\Externalization{G}\) are identities.

With the data previously defined, we have the following proposition.

\begin{proposition}
	\label{prop:V-cat-index-cat}
	There is a 2-functor \(\Externalization{\MathDash}\) from \(\Cat[\E][\IntV]\) to the 2-category of indexed \(\Externalization{\IntV}\)-categories \(\ICat[\E][\Externalization{\IntV}]\).
\end{proposition}

The previous result is extremely weak.
Indeed, we would like to better understand the 2-functor \(\Externalization{\MathDash}\).

Observe that there is a construction inducing \(\IntV\)-enriched functors from indexed \(\Externalization{\IntV}\)-functors.
Let \(\mathcal{F} \colon \Externalization{\X} \to \Externalization{\Y}\) be a \(\Externalization{\IntV}\)-functor.
Define the \(\IntV\)-functor \(\bar{\mathcal{F}} \colon \X \to \Y\) as
\begin{align*}
	\bar{\mathcal{F}}_0 &\DefEq {(\mathcal{F}^X)}_0(\Id(X)) \colon X \to Y \\
	\bar{\mathcal{F}}_1 &\DefEq \phi \Comp[\Y] {(\mathcal{F}^{X \Prod X})}_1(\pi_1, \pi_2) \colon X \Prod X \to V_1.
\end{align*}
The isomorphism \(\phi\) appearing in the definition of the morphism component requires some explanation.
The source and target of \({(\mathcal{F}^{X \Prod X})}_1(\pi_1, \pi_2)\) are, respectively, \({(\mathcal{F}^{X \Prod X})}_0(\pi_1)\) and \({(\mathcal{F}^{X \Prod X})}_0(\pi_2)\), while we need an arrow from \({(\mathcal{F}^{X})}_0(\Id(X)) \pi_1\) to \({(\mathcal{F}^{X})}_0(\Id(X)) \pi_2\) to match the definition of \(\bar{\mathcal{F}}\) on objects.
We fix this issue by introducing a suitable isomorphism.
By the definition of \(\Externalization{\IntV}\)-functor, we have an isomorphism
\[
	\mathcal{F}^{X \Prod X} \Comp {\pi_i}^* \Iso {\pi_i}^* \Comp {({\pi_i}^*)}_\bullet (\mathcal{F}^X) \colon {({\pi_i}^*)}_\bullet ( \Externalization{\X}[X] ) \to \Externalization{\X}[X \Prod X]
\]
which we apply to the object \(\Id(X)\) of \({({\pi_i}^*)}_\bullet ( \Externalization{\X}[X] )\) to get an isomorphism
\[
	\phi_i \colon {(\mathcal{F}^{X \Prod X})}_0(\pi_i) \Iso {(\mathcal{F}^{X})}_0(\Id(X)) \pi_i
\]
in \(\Externalization{\X}[X \Prod X]\).
From \(\phi_1\) and \(\phi_2\) we get the isomorphism \(\phi\) that we need.

Moreover, there is also a construction inducing \(\IntV\)-enriched natural transformations from indexed \(\Externalization{\IntV}\)-natural transformations.
Let \(\alpha \colon \mathcal{F} \to \mathcal{G} \colon \Externalization{\X} \to \Externalization{\Y}\) be an indexed \(\Externalization{\IntV}\)-natural transformation.
Define the \(\IntV\)-enriched natural transformation \(\bar{\alpha} \colon \bar{\mathcal{F}} \to \bar{\mathcal{G}} \colon \X \to \Y\) as
\(\bar{\alpha} \DefEq \alpha^X(\Id(X)) \colon X \to V_1\).

It is clear that, for a \(\IntV\)-enriched functor \(F \colon \X \to \Y\), we have \(F = \overline{\Externalization{F}}\), and for a \(\IntV\)-enriched natural transformation \(\alpha \colon F \to G \colon \X \to \Y\), we have \(\alpha = \overline{\Externalization{\alpha}}\).
This provides a strengthening of \Cref{prop:V-cat-index-cat}, in that it shows that \(\Cat[\E][\IntV]\) is a sub-2-category of \(\ICat[\E][\Externalization{\IntV}]\).
Moreover, as proved in the following proposition, \(\Cat[\E][\IntV]\) is a full sub-2-category, meaning that there is an equivalence of hom-categories.

\begin{proposition}
	\label{prop:V-cat-index-cat-fullsubcat}
The 2-category \(\Cat[\E][\IntV]\) is a full sub-2-category of \(\ICat[\E][\Externalization{\IntV}]\), and \(\Externalization{\MathDash}\) is the relative inclusion.
\end{proposition}
\begin{proof}
	Consider indexed \(\Externalization{\IntV}\)-categories \(\Externalization{\X}\) and \(\Externalization{\Y}\).
	We need to show that there is an equivalence of categories
	\[
		\Cat[\IntV](\X, \Y) \equiv \Cat[\Externalization{\IntV}](\Externalization{\X}, \Externalization{\Y}).
	\]	
	
	Let \(\mathcal{F} \colon \Externalization{\X} \to \Externalization{\Y}\) be a \(\Externalization{\IntV}\)-functor, and consider the indexed \(\Externalization{\IntV}\)-functor \(\Externalization{\bar{\mathcal{F}}} \colon \Externalization{\X} \to \Externalization{\Y}\).
	We need to prove that there is a \(\Externalization{\IntV}\)-natural isomorphism \(\mathcal{F} \Iso \Externalization{\bar{\mathcal{F}}}\).
	Let \(I\) be an indexing object in \(\E\).
	Then we need a natural isomorphism \(\mathcal{F}^I \Iso \Externalization{\bar{\mathcal{F}}}[I]\).
	Let \(x \colon I \to X\) be an object of \(\Externalization{\X}[I]\).
	By the definition of \(\Externalization{\IntV}\)-functor, we have an isomorphism
	\[
		\mathcal{F}^I \Comp x^* \Iso x^* \Comp {(x^*)}_\bullet (\mathcal{F}^X) \colon {(x^*)}_\bullet ( \Externalization{\X}[X] ) \to \Externalization{\X}[I]
	\]
	which we apply to the object \(\Id(X)\) of \({(x^*)}_\bullet ( \Externalization{\X}[X] )\) to get an isomorphism
	\[
		\mathcal{F}^I_0(x) \Iso_x \mathcal{F}^X_0(\Id(X)) \Comp x
	\]
	in \(\Externalization{\X}[I]\).
	Then, we take that as the definition of the isomorphism \(\mathcal{F}^I \Iso \Externalization{\bar{\mathcal{F}}}[I]\) on \(x \colon I \to X\).
	
	We need to prove that the isomorphism just defined is natural.
	Let \(I\) be an indexing object in \(\E\) and \(x_1, x_2 \colon I \to X\) objects of \(\Externalization{\X}[I]\).
	By the definition of \(\Externalization{\IntV}\)-functor, we have an isomorphism
	\[
		\mathcal{F}^I \Comp {(x_1, x_2)}^* \Iso {(x_1, x_2)}^* \Comp {({(x_1, x_2)}^*)}_\bullet (\mathcal{F}^{X \Prod X}) \colon {({(x_1, x_2)}^*)}_\bullet ( \Externalization{\X}[X \Prod X] ) \to \Externalization{\X}[I]
	\]
	which we apply to the objects \(\pi_i\) of \({({(x_1, x_2)}^*)}_\bullet ( \Externalization{\X}[X \Prod X] )\) to get an isomorphism
	\[
		\mathcal{F}^I_0(x_i) \Iso \mathcal{F}^{X \Prod X}_0(\pi_i) \Comp {(x_1, x_2)}
	\]
	in \(\Externalization{\X}[I]\).
	Then, consider the following diagram.
	\[
		\begin{tikzcd}
			\mathcal{F}^I_0(x_1)
					\ar[r, "{\Iso}"]
					\ar[d, "{\mathcal{F}^I_1(x_1, x_2)}"']
				&\mathcal{F}^{X \Prod X}_0(\pi_1) \Comp {(x_1, x_2)}
					\ar[r, "{\Iso}"]
					\ar[d, "{\mathcal{F}^{X \Prod X}_1(\pi_1, \pi_2) \Comp {(x_1, x_2)}}"]
				&\mathcal{F}^{X}_0(\Id(X)) \Comp x_1
					\ar[d, "{\Externalization{\bar{\mathcal{F}}}[I]_1(x_1, x_2)}"] \\
			\mathcal{F}^I_0(x_2)
					\ar[r, "{\Iso}"]
				&\mathcal{F}^{X \Prod X}_0(\pi_2) \Comp {(x_1, x_2)}
					\ar[r, "{\Iso}"]
				&\mathcal{F}^{X}_0(\Id(X)) \Comp x_1
		\end{tikzcd}
	\]
	Firstly, the left-hand-side square commutes because of the naturality of the isomorphism.
	Secondly, the right-hand-side square commutes because such is the definition of \(\Externalization{\bar{\mathcal{F}}}\).
	Finally, the composition of the consecutive isomorphisms \(\mathcal{F}^I_0(x_i) \to \mathcal{F}^{X}_0(\Id(X)) \Comp x_i\) is the isomorphism \(\mathcal{F}^I \Iso \Externalization{\bar{\mathcal{F}}}[I]\) computed on \(x_i\), because of the functoriality axiom for \(\Externalization{\IntV}\)-functors applied to the functor \(\mathcal{F}\) and the composition \(\pi_i \Comp (x_1, x_2) = x_i\).
	But then the outer square is the naturality diagram, and we have shown that it commutes.
	
	Then, we need to prove that the isomorphism \(\mathcal{F} \Iso \Externalization{\bar{\mathcal{F}}}\) satisfies the naturality condition for \(\Externalization{\IntV}\)-natural transformations.
	Let \(f \colon I \to J\) be a reindexing and \(x \colon J \to X\) an object of \(\Externalization{\X}\).
	Then the naturality diagram for the reindexing \(f\) and computed on \(x\) is
	\[
		\begin{tikzcd}
			\mathcal{F}^I_0(j \Comp f)
					\ar[dr, "{\Iso_{x \Comp f}}"']
					\ar[rr, "{\Iso_{f}}"]
				&&\mathcal{F}^I_0(x) \Comp f
					\ar[dl, "{\Iso_{x} \Comp f}"] \\
			&\Externalization{\bar{\mathcal{F}}}[I]_0(x \Comp f) = \mathcal{F}^X_0(\Id(X)) \Comp x \Comp f
		\end{tikzcd}	
	\]
	which commutes thanks to the functoriality axiom for \(\Externalization{\IntV}\)-functors applied to the functor \(\mathcal{F}\) and the composition of \(f\) and \(x\).

	Finally, we need to prove that, for any indexed \(\Externalization{\IntV}\)-natural transformation \(\alpha \colon \mathcal{F} \to \mathcal{G} \colon \Externalization{\X} \to \Externalization{\Y}\), the following square commute.
	\[
		\begin{tikzcd}
			\mathcal{F}
					\ar[r, "{\alpha}"]
					\ar[d, "{\Iso}"']
				&\mathcal{G}
					\ar[d, "{\Iso}"] \\
			\Externalization{\bar{\mathcal{F}}}
					\ar[r, "{\Externalization{\bar{\alpha}}}"]
				&\Externalization{\bar{\mathcal{G}}}
		\end{tikzcd}	
	\]
	Let \(I\) be an indexing object in \(\E\) and \(x \colon I \to X\) an object of \(\Externalization{\X}[I]\), and compute the \(I\)-th component of the above diagram on \(x\).
	We get a commutative square, as it is an instance of the naturality axiom for the indexed \(\Externalization{\IntV}\)-natural transformation \(\alpha\), relative to the reindexing \(x\) and computed on \(\Id(X)\).	
\end{proof}

\begin{remark}
	The converse of \Cref{prop:V-cat-index-cat-fullsubcat} does not seem to hold, that is, indexed \(\Externalization{\IntV}\)-categories don't canonically induce internal \(\IntV\)-enriched categories.
	In particular, the categories \(\Externalization{\X}[I]\) are small, as their object of objects is the homset \(\E(I, X)\), but that is not generally the case for indexed \(\Externalization{\IntV}\)-categories.
\end{remark}

We then straightforwardly get the following corollary, which says something about enriched indexed categories.
It is indeed generally false that small \(\IdxW\)-categories are indexed \(\IdxW\)-categories too, but that happens to be the case when the enriching indexed category is the externalization of an internal monoidal category.

\begin{corollary}
	The 2-category \(\SCat[\E][\Externalization{\IntV}]\) is a full sub-2-category of \(\ICat[\E][\Externalization{\IntV}]\).
\end{corollary}

To conclude, we look at the interplay between externalization and the underlying category of points.

\begin{proposition}
	Let \(\X\) be a \(\IntV\)-enriched category.
	Then, there is a natural isomorphism of indexed categories \(U \big( \Externalization{\X} \big) \Iso \Externalization{U(\X)}\)
	between the underlying indexed category of the indexed \(\Externalization{\IntV}\)-enriched category \(\Externalization{\X}\) and the externalization of the underlying \(\E\)-category \(U(\X)\).
\end{proposition}
\begin{proof}
	Let \(I\) be an indexing object.
	We need to prove that there is an isomorphism \(U \big( \Externalization{\X}[I] \big) \Iso \Externalization{U(\X)}[I]\) between the underlying standard category of the \(\Externalization{\IntV}[I]\)-enriched category \(\Externalization{\X}[I]\) and the fiber over \(I\) of the externalization of the underlying \(\E\)-category \(U(\X)\).
	Moreover, for any reindexing \(f \colon J \to I\) in \(\E\), the square
	\[
		\begin{tikzcd}
			U \big(\Externalization{\X}[I] \big)
				\ar[r, "{\Iso}"]
				\ar[d, "{U \big( f^*(\Externalization{\X}) \big)}"']
			&\Externalization{U(\X)}[I]
				\ar[d, "{f^* \big( \Externalization{U(\X)} \big)}"] \\
			U \big(\Externalization{\X}[J] \big)
				\ar[r, "{\Iso}"]
			&\Externalization{U(\X)}[J]
		\end{tikzcd}
	\]
	has to commute.
	
	For both categories, the objects are \(I\)-indexed families of objects of \(\X\), and the arrows \((I \xrightarrow{x_0} X) \to (I \xrightarrow{x_1} X)\) are the sections of the projection
	\[
		{\big( \MonUnit[\IntV] !_X, \Hom[\X](x_0, x_1) \big)}^* V_1 \to I,
	\]
	so that the categories are clearly isomorphic to each other, and the square commutes trivially.
\end{proof}

The previous result can be extended to the following proposition.

\begin{proposition}
	The following diagram of 2-functors	commutes.
	\[
		\begin{tikzcd}
			\Cat[\E][\IntV]
					\ar[r, "{U}"]
					\ar[d, "{\Externalization{\MathDash}}"']
				&\Cat[\E]
					\ar[d, "{\Externalization{\MathDash}}"] \\
			\ICat[\E][\Externalization{\IntV}]
					\ar[r, "{U}"]
				&\ICat[\E]
		\end{tikzcd}
	\]
\end{proposition}

The above discussion suggests that the indexed \(\Externalization{\IntV}\)-category \(\Externalization{\X}\) should yield a notion of externalization of an internal enriched category \(\X\).
That would be defined as the large \(\Externalization{\IntV}\)-category obtained by \(\Externalization{\X}\) via the functor \(\Theta\) \autocite[Section~6]{shulman2013enriched}, analogously to how one gets the total category of an indexed category via the Grothendieck construction (\Cref{thm:groth-constr}).

\section{Enriched generalized multicategories}\label{sec:enriched-gen-multicats}

In this section we show that the notion of internally enriched category is an occurrence of Leinster's more general notion of enriched generalized multicategory \autocite{leinster1999generalized,Leinster04HigherOperads}.
For that, we shall follow the exposition in \textcite{Leinster04HigherOperads} and adopt its notation.

In addition to the usual assumptions on \(\E\), we need to assume that the forgetful functor \(U \colon \Cat[\E] \to \Graph[\E]\) has a left adjoint, the free-category-on-a-graph functor \(\FreeCat \colon \Graph[\E] \to \Cat[\E]\).
In particular, that is true if \(\E\) has countable products and finite coproducts.
The adjunction \((\FreeCat, U, \Unit, \Counit)\) is monadic, and it induces the free-category-over-a-graph monad \((\FreeCat = U\FreeCat, \FreeCatUnit, \FreeCatMult = U \Counit \FreeCat)\) over \(\Graph[\E]\) \autocite[theorem {6.5.2}]{Leinster04HigherOperads}.

\textcite{Leinster04HigherOperads} presents a general notion of \(T\)-multicategory, for \(T\) a cartesian monad, but in our case we are only interested in \(\FreeCat\)-multicategories.
For simplicity's sake, then, we shall state the specialized definition rather then the general one.

\begin{definition}[generalized multicategory, {\autocite[definition {4.2.2}]{Leinster04HigherOperads}}]
	A \Def{\(\FreeCat\)-graph} \(\MC{C}\) is a diagram
	\[
		\begin{tikzcd}
			&\IntC_1
				\ar[dl, "{\MCSource[\MC{C}]}"']
				\ar[dr, "{\MCTarget[\MC{C}]}"] \\
			\FreeCat \IntC_0
			& & \IntC_0
		\end{tikzcd}
	\]
	in \(\Graph[\E]\) (so, sort to say, a diagram of diagrams).
	A \(\FreeCat\)-multicategory \(\MC{C}\) consists of a \(\FreeCat\)-graph together with arrows
	\begin{gather*}
		\MCComp[\MC{C}] \colon \IntC_1 \Pullback[][\FreeCat \IntC_0] \FreeCat \IntC_1 \to \IntC_1
	\shortintertext{and}
		\MCId[\MC{C}] \colon \IntC_0 \to \IntC_1
	\end{gather*}
	satisfying associativity and unitarity axioms.
	More abstractly, a \(\FreeCat\)-multicategory is a monad in the category of \(\FreeCat\)-algebras.

A \Def{functor of \(\FreeCat\)-categories} \(F \colon \MC{C} \to \MC{D}\) is given by a morphism of the underlying graphs, that is, a commutative diagram
\begin{equation*}
	\label{eqn:funct-fc-cat}
	\begin{tikzcd}[column sep = small]
		&\IntC_1
			\ar[dl, "{\MCSource[\MC{C}]}"']
			\ar[dr, "{\MCTarget[\MC{C}]}"]
			\ar[dd, "{F_1}"] \\
		\FreeCat \IntC_0
			\ar[dd, "{\FreeCat F_0}"']
		& & \IntD_0
			\ar[dd, "{F_0}"]  \\
		&\IntD_1
			\ar[dl, "{\MCSource[\MC{D}]}"']
			\ar[dr, "{\MCTarget[\MC{D}]}"] \\
		\FreeCat \IntD_0
		& & \IntD_0
	\end{tikzcd}
\end{equation*}
in \(\Graph[\E]\),
satisfying both the identity and composition functoriality axioms.
\end{definition}

The objects of \(\E\) have a canonical \(\FreeCat\)-multicategory structure.
Indeed, if \(X\) is an object of \(\E\), the (underlying graph of the) indiscrete internal category \(\Ind(X)\) in \(\E\) has an \(\FreeCat\)-algebra structure
\[
	\Comp[\Ind(X)] = U \Counit_{\Ind(X)} \colon \FreeCat\Ind(X) \to \Ind(X)
\]
given by the counit of the adjunction.
Then there is a (unique) \(\FreeCat\)-multicategory structure \(\IndP[X]\) \autocite[example {4.2.22}]{Leinster04HigherOperads} given by a \(\FreeCat\)-graph
\[
	\begin{tikzcd}[column sep = small, ampersand replacement=\&]
		\& \FreeCat\Ind(X)
			\ar[dl, "{\MCSource[\IndP[X]] = \Id(\FreeCat\Ind(X))}"']
			\ar[dr, "{\MCTarget[\IndP[X]] = \Comp[\Ind(X)]}"] \\
		\FreeCat\Ind(X)
			\& \& \Ind(X)
	\end{tikzcd}
\]
and operations
\begin{gather*}
	\MCComp[\IndP[X]] = \FreeCat\Ind(X) \Pullback[][\FreeCat\Ind(X)] \FreeCat\FreeCat\Ind(X)
		\xrightarrow{\pi_2}
	\FreeCat\FreeCat\Ind(X)
		\xrightarrow{\FreeCatMult[\Ind(X)]} \FreeCat\Ind(X) \\
	\MCId[\IndP[X]] = \Ind(X) \xrightarrow{\FreeCatUnit[\Ind(X)]} \FreeCat\Ind(X)
\end{gather*}
satisfying associativity and unitarity.
Intuitively, the composition of \(\IndP[X]\) collapses a list of lists of elements of \(X\) into a single list:
\[
	(x_0, \dots, x_n) \MCComp[\IndP[X]] \big( (x_0, \dots, x_1), \dots, (x_{n-1}, \dots, x_n) \big) = (x_0, \dots, x_1, \dots, x_{n-1}, \dots, x_n)
\]

We can now define a notion of enrichment over a \(\FreeCat\)-multicategory.

\begin{definition}[enriched generalized multicategory, {\autocite[Definition~{6.8.1}]{Leinster04HigherOperads}}]
Let \(\MCV\) be an \(\FreeCat\)-multicategory.
A \Def{\(\MCV\)-enriched \(\IdMon[\E]\)-multicategory} is given by an object \(X\) of \(\E\) together with a map \(\IndP[X] \to \MCV\) of \(\FreeCat\)-multicategories.
\end{definition}

We shall now induce a \(\FreeCat\)-multicategory \(\MCV\) (a \(\FreeCat\)-operad, to be more specific) from the monoidal category \(\IntV\) in \(\E\).
The construction is designed so that \(\MCV\)-enriched \(\IdMon\)-multicategories shall correspond to \(\IntV\)-enriched internal categories.

Let \(\IntV_0\) be the unique \(\E\)-graph \(\Terminal \leftleftarrows V_{0}\).
Consider the category in \(\E\) on \(\IntV_0\) given by the monoidal structure on \(\IntV\), i.e., having composition \(\Comp[\IntV_0] \DefEq \MonProd[\IntV]\) and identity \(\Id[\IntV_0] \DefEq \MonUnit[\IntV]\).
Then, by adjunction, there is a functor \(\MonProd[0] \colon \FreeCat(\IntV_0) \to \IntV_0\) of categories in \(\E\)
induced by the  morphism of \(\E\)-graphs \(\Id(\IntV_0)\) and which is trivial on the object's component.

Let \(\IntV_1\) be the graph \(V_{0}^{+} \leftleftarrows V_{1}^{+}\) given as the pullback of \(\E\)-graphs
\[
	\begin{tikzcd}
		\IntV_1
				\ar[r, "\pi_2"]
				\ar[d, "\pi_1"']
				\ar[dr, phantom, "\lrcorner" very near start]
			& (\Terminal \leftleftarrows V_1)
				\ar[d, "{\Source[\IntV]}"] \\
		\FreeCat(\IntV_0) \ar[r, "{\MonProd[0]}"']
			& \IntV_0
	\end{tikzcd}
	=
	\left(
	\begin{tikzcd}
		V_{0}^{+}
				\ar[r, "\pi_2"]
				\ar[d, "\pi_1"']
				\ar[dr, phantom, "\lrcorner" very near start]
			& \Terminal
				\ar[d] \\
		\Terminal
				\ar[r]
			& \Terminal
	\end{tikzcd}
	\leftleftarrows
	\begin{tikzcd}
		V_{1}^{+}
				\ar[r, "\pi_2"]
				\ar[d, "\pi_1"']
				\ar[dr, phantom, "\lrcorner" very near start]
			& V_1 \ar[d, "{\Source[\IntV]}"] \\
		{\FreeCat(\IntV_0)}_1
				\ar[r, "{\MonProd[0]}"']
			& V_0
	\end{tikzcd}
	\right).
\]
Notice that \(V_{0}^{+} = \Terminal\).
In the internal language of \(\E\), the pullback \(V_{1}^{+}\) has elements those pairs \((v_i, f) \colon {\FreeCat(\IntV_0)}_1 \Prod V_1\) such that \(\Source[\IntV](f) = \MonProd[0](v_i)\).
Thus, informally speaking, \(\IntV_1\) is the graph with one vertex and arrows \(f \colon v_0 \MonProd[\IntV] \dots \MonProd[\IntV] v_n \to v\) of \(\IntV\) as edges, and also tracking the sequence \(v_0, \dots, v_n\) which yields the domain of \(f\).
There is an \(\E\)-category structure on \(\IntV_1\) with composition defined as
\[
	\begin{tikzcd}
		V_{1}^{+} \Prod V_{1}^{+}
				\ar[rr, "{\pi_2 \Prod \pi_2}"]
				\ar[dr, dashed, "{\Comp[\IntV_1]}" description]
				\ar[dd, "{\pi_1 \Prod \pi_1}"']
		 	& & V_1 \Prod V_1
				\ar[d, "{\MonProd[\IntV]}"] \\
			& V_{1}^{+}
				\ar[r, "\pi_2"]
				\ar[d, "\pi_1"']
				\ar[dr, phantom, "\lrcorner" very near start]
			& V_1 \ar[d, "{\Source[\M]}"] \\
		{\FreeCat(\IntV_0)}_1 \Prod {\FreeCat(\IntV_0)}_1
				\ar[r, "{\Comp[\FreeCat(\IntV_0)]}"']
			& {\FreeCat(\IntV_0)}_1
				\ar[r, "{\MonProd[0]}"']
			& V_0
	\end{tikzcd}
\]
and identity defined as follows.
\begin{equation*}
	\begin{tikzcd}
		\Terminal
				\ar[rrd, bend left, "{\Id[\IntV] \MonUnit[\IntV]}"]
				\ar[dr, dashed, "{\Id[\IntV_1]}" description]
				\ar[dd, "{\MonUnit[\IntV]}"'] \\
		& V_{1}^{+}
				\ar[r, "\pi_2"]
				\ar[d, "\pi_1"']
				\ar[dr, phantom, "\lrcorner" very near start]
			& V_1
				\ar[d, "{\Source[\IntV]}"] \\
		V_0
				\ar[r, "{{(\FreeCatUnit[\IntV_0])}_1}"']
			& {\FreeCat(\IntV_0)}_1
				\ar[r, "{\MonProd[0]}"']
			& V_0
	\end{tikzcd}
\end{equation*}
In the internal language, the composition is defined as
\[
	(v_i, f), (w_j, g) \colon V_1^+ \TiC (w_j, g) \Comp[\IntV_1] (v_i, f) \DefEq (v_i \Comp[\FreeCat(\IntV_0)] w_j, f \MonProd[\IntV] g)
\]
or, informally speaking,
\[
	\big( v_0 \MonProd \dots \MonProd v_n \xrightarrow{f} v \big)
		\Comp[\IntV_1]
	\big( w_0 \MonProd \dots \MonProd w_m \xrightarrow{g} w \big)
		\DefEq
	\big( v_0 \MonProd \dots \MonProd v_n \MonProd w_0 \MonProd \dots \MonProd w_m \xrightarrow{f \MonProd g} v \MonProd w \big).
\]
Then by adjunction there is a functor
\begin{equation*}
	\MonProd[1] \colon \FreeCat(\IntV_1) \to \IntV_1
\end{equation*}
of categories in \(\E\) induced by the morphism of \(\E\)-graphs \(\Id(\IntV_1)\) and which is trivial on objects.

We can finally define the \(\FreeCat\)-multicategory \(\MCV\) induced by \(\IntV\).
Its underlying \(\FreeCat\)-graph is the following.
\[
	\begin{tikzcd}[column sep = small]
		& \IntV_1
			\ar[dl, "{\MCSource[\MCV] = \pi_1}"']
			\ar[dr, "{\MCTarget[\MCV] = \Target[\IntV] \pi_2}"] \\
		\FreeCat(\IntV_0)
			& & \IntV_0
	\end{tikzcd}
\]
The composition \(\MCComp[\MCV] \colon \IntV_1 \Pullback[][\FreeCat(\IntV_0)] \FreeCat(\IntV_1) \to \IntV_1\) is defined by
\begin{equation*}
	\begin{tikzcd}
		\IntV_1 \Pullback[][\FreeCat(\IntV_0)] \FreeCat(\IntV_1)
				\ar[r, "{\IntV_1 \Pullback[][\IntV_0] \MonProd[1]}"]
				\ar[dr, dashed, "{\MCComp[\MCV]}" description]
				\ar[d, "{\pi_2}"']
			& \IntV_1 \Pullback[][\IntV_0] \IntV_1
				\ar[r, "{\pi_2 \Prod \pi_2}"]
			& (\Terminal \leftleftarrows V_1 \Pullback[][V_0] V_1)
				\ar[d, "{\Comp[\IntV]}"] \\
		\FreeCat(\IntV_1)
				\ar[d, "{\FreeCat(\pi_1)}"']
			& \IntV_1
				\ar[r, "\pi_2"]
				\ar[d, "\pi_1"']
				\ar[dr, phantom, "\lrcorner" very near start]
			& (\Terminal \leftleftarrows V_1 \ar[d, "{\Source[\IntV]}"]) \\
		\FreeCat\FreeCat(\IntV_0)
				\ar[r, "{\FreeCatMult[\IntV_0]}"']
			& \FreeCat(\IntV_0)
				\ar[r, "{\MonProd[0]}"']
			& \IntV_0
	\end{tikzcd}
\end{equation*}
and the identity \(\MCId[\MCV] \colon \IntV_0 \to \IntV_1\) is defined by the following diagram.
\begin{equation*}
	\begin{tikzcd}
		\IntV_0
				\ar[rrd, bend left, "{\Id[\IntV]}"]
				\ar[dr, dashed, "{\MCId[\MCV]}" description]
				\ar[ddr, bend right, "{\FreeCatUnit[\IntV_0]}"'] \\
		& \IntV_1
				\ar[r, "\pi_2"]
				\ar[d, "\pi_1"']
				\ar[dr, phantom, "\lrcorner" very near start]
			& (\Terminal \leftleftarrows V_1 \ar[d, "{\Source[\IntV]}"]) \\
		& \FreeCat(\IntV_0)
				\ar[r, "{\MonProd[0]}"']
			& \IntV_0
	\end{tikzcd}
\end{equation*}
In the internal language, given \((A_i, f) \colon V_1^+\) and \((A_j^i, f_i) \colon {\FreeCat(\IntV_1)}_1\) such that \(A_i = \Target[\IntV](f_i) \colon {\FreeCat(\IntV_0)}_1\),
the multi-category composition on the arrows' component is defined as
\[
	(A_i, f) \MCComp[\MCV] (A_j^i, f_i) \DefEq \Big( \FreeCatMult[\IntV_0] (A_j^i), f \Comp[\IntV] \big(\MonProd[1](f_i)\big) \Big) \colon V_1^+
\]
(and it is trivial on the object component).
Informally, \(f \MCComp[\MCV] (f_0, \dots, f_n)\) is represented as follows.
\[
	\begin{tikzcd}[column sep = 0]
		A^0_0 \MonProd \dots \MonProd A^0_{m_0}
			\ar[d, "{f_0}"]
		& {}
			\ar[d, phantom, "{\MonProd \dots \MonProd}" description]
		& A^n_0 \MonProd \dots \MonProd A^n_{m_n}
			\ar[d, "{f_n}"'] \\
		A_0
		& {}
			\ar[d, "{f}"]
		& A_n \\
		& A
	\end{tikzcd}
\]

Now that we have an enriching multicategory, we can give the two inverse constructions relating enriched generalized multicategories and internal enriched categories.

First, we shall build a \(\MCV\)-enriched \(\IdMon\)-multicategory \(\MCX\) out of a \(\IntV\)-enriched category \(\X\).
That is given by an object of \(\E\), which we choose to be the underlying object of objects \(X\) of \(\X\), and a morphism \(\IndP[X] \to \MCV\) of \(\FreeCat\)-categories; that means a commutative diagram in \(\Graph[\E]\)
\begin{equation}
	\label{eqn:M-enrchd-id-cat}
	\begin{tikzcd}[column sep = small]
		& \FreeCat\Ind(X)
				\ar[dl, equal]
				\ar[dr, "{\Comp[\Ind(X)]}"]
				\ar[dd, "{\MCX_1}"] \\
		\FreeCat\Ind(X)
				\ar[dd, "{\FreeCat\MCX_0}"']
			& & \Ind(X)
				\ar[dd, "{\MCX_0}"] \\
		& \IntV_1
				\ar[dl, "{\Source[\MCV]}"']
				\ar[dr, "{\Target[\MCV]}"] \\
		\FreeCat(\IntV_0)
			& & \IntV_0
	\end{tikzcd}
\end{equation}
satisfying the functoriality axioms.
Let \(\MCX_0\) be the morphism \(\Ind(X) \to \IntV_0\) of graphs in \(\E\) given by \(\Hom[\X]\) on edges.
The arrow \(\MCX_1 \colon \FreeCat\Ind(X) \to \IntV_1\) shall require a two-step construction.

Consider the \(\E\)-graph \(\IntV_{\X} \colon {(\IntV_{\X})}_0 \leftleftarrows {(\IntV_{\X})}_1\), given by
\begin{equation*}
	\begin{tikzcd}
		\IntV_{\X}
				\ar[r, "\pi_2"]
				\ar[d, "\pi_1"']
				\ar[dr, phantom, "\lrcorner" very near start]
			& \IntV_1
				\ar[d, "{\Target[\IntV]\pi_2}"] \\
		\Ind(X)
				\ar[r, "{\MCX_0}"']
			& \IntV_0
	\end{tikzcd}
	=
	\left(
	\begin{tikzcd}
		{(\IntV_{\X})}_0
				\ar[r, "\pi_2"]
				\ar[d, "\pi_1"']
				\ar[dr, phantom, "\lrcorner" very near start]
			& \Terminal
				\ar[d] \\
		X
				\ar[r]
			& \Terminal
	\end{tikzcd}
	\leftleftarrows
	\begin{tikzcd}
		{(\IntV_{\X})}_1
				\ar[r, "\pi_2"]
				\ar[d, "\pi_1"']
				\ar[dr, phantom, "\lrcorner" very near start]
			& V_{1}^{+}
				\ar[d, "{\Target[\IntV] \pi_2}"] \\
		X \Prod X
				\ar[r, "{\Hom[\X]}"']
			& V_0
	\end{tikzcd}
	\right)
\end{equation*}
and notice that \({(\IntV_{\X})}_0 = X\), while, in the internal language, \({(\IntV_{\X})}_1\) has elements those tuples \((x_0 \colon X, x_1 \colon X, (v_i) \colon {\FreeCat(\IntV_0)}_1, f \colon V_1)\) such that \(\Source[\IntV](f) = \MonProd[0](v_i)\) and \(\Target[\IntV](f) = \Hom[\X](x_0, x_1)\).
Informally, \(\IntV_{\X}\) is the graph with elements of \(X\) as vertices, and arrows \(f \colon v_0 \MonProd \dots \MonProd v_n \to \Hom[\X](x_0, x_1)\) as edges from the vertex \(x_0\) to the vertex \(x_1\).
The graph \(\IntV_{\X}\) has the structure of an internal category in \(\E\).
Its composition \(\Comp[\IntV_{\X}] \colon {(\IntV_{\X})}_1 \Pullback[][X] {(\IntV_{\X})}_1 \to {(\IntV_{\X})}_1\) is the unique arrow defined by the following diagram.
\begin{equation*}
	\begin{tikzcd}[row sep = large]
		V_1
			& V_{1} \Pullback[][V_0] V_{1}
				\ar[l, "{\pi_2}"']
				\ar[d, "{\pi_1}"]
				\ar[r, "{\Comp[\IntV]}"]
			& V_{1} \\
		V_{1} \Prod V_{1}
				\ar[r, "{\MonProd[\IntV]}"]
			& V_1 \\
		V_{1}^{+} \Prod V_{1}^{+}
				\ar[u, "{\pi_2 \Prod \pi_2}"]
				\ar[r, "{\pi_1 \Prod \pi_1}"]
			& \FreeCat(\IntV_0)_1 \Prod \FreeCat(\IntV_0)_1
				\ar[r, "{\Comp[\FreeCat(\IntV_0)]}"]
			& \FreeCat(\IntV_0)_1 \\
		{(\IntV_{\X})}_1 \Pullback[][X] {(\IntV_{\X})}_1
				\ar[u, "{{(\pi_2)}_1 \Prod {(\pi_2)}_1}"]
				\ar[urr, dashed]
				\ar[d, "{{(\pi_1)}_1 \Pullback[][X] {(\pi_1)}_1}"']
				\ar[r, dashed, "{\Comp[\IntV_{\X}]}" description]
			& {(\IntV_{\X})}_1
				\ar[r, "{(\pi_2)}_1"]
				\ar[d, "{(\pi_1)}_1"']
				\ar[dr, phantom, "\lrcorner" very near start]
			& V_{1}^{+}
				\ar[uuu, bend right = 45, "{\pi_2}"]
				\ar[u, "{\pi_1}"]
				\ar[d, "{\Target[\IntV] \pi_2}"] \\
		X \Prod X \Prod X
				\ar[r, "{(\pi_1, \pi_3)}"']
				\ar[uuuu, bend left = 90, "{\Comp[\X]}"']
			& X \Prod X \ar[r, "{\Hom[\X]}"']
			& V_0
	\end{tikzcd}
\end{equation*}
Observe that there is a morphism \(h \colon X \to V_{1}^{+}\) defined via the universal property of the pullback by the arrows \(\FreeCatUnit(V_0) \MonUnit[\IntV] ! \colon X \to {\FreeCat(\IntV_0)}_1\) and \(\Id[\X] \colon X \to V_1\).
Then, the identity \(\Id[\IntV_{\X}] \colon X \to {(\IntV_{\X})}_1\) is given as the unique arrow defined by the following diagram.
\begin{equation*}
	\begin{tikzcd}
		X
			\ar[drr, bend left, "{h}"]
			\ar[ddr, bend right, "{\Delta_{X}}"'] \ar[dr, dashed, "{\Id[\IntV_{\X}]}" description] \\
		& {(\IntV_{\X})}_1
				\ar[r, "{(\pi_2)}_1"]
				\ar[d, "{(\pi_1)}_1"']
				\ar[dr, phantom, "\lrcorner" very near start]
			& V_{1}^{+}
				\ar[d, "{\Target[\IntV] \pi_2}"] \\
		& X \Prod X
				\ar[r, "{\Hom[\X]}"']
			& V_0
	\end{tikzcd}
\end{equation*}
In the internal language, the composition is given by
\[
	(x_1, x_2, B_j, g) \Comp[\IntV_{\X}] (x_0, x_1, A_i, f) \DefEq (x_0, x_2, A_i \Comp[\FreeCat(\IntV_0)] B_j, (\Comp[\X](x_0, x_1, x_2)) \Comp[\IntV] (f \MonProd[\IntV] g)).
\]
Informally, the composition is given by composing the monoidal product of \(f\) and \(g\) with the composition morphism of \(\X\), as shown in the diagram below.
\[
	\begin{tikzcd}[column sep = 0]
		B_0 \MonProd \dots \MonProd B_m
			\ar[d, "{g}"]
		& {} \ar[d, phantom, "{\MonProd}" label]
		& A_0 \MonProd \dots \MonProd A_n
			\ar[d, "{f}"'] \\
		\Hom[\X](x_1, x_2)
		& {}
			\ar[d, "{\Comp[\X](x_0, x_1, x_2)}"]
		& \Hom[\X](x_0, x_1) \\
		& \Hom[\X](x_0, x_2)
	\end{tikzcd}
\]

Internal functors \(\FreeCat\Ind(X) \to \IntV_{\X}\) in \(\E\) correspond bijectively, by adjointness, to morphisms \(\Ind(X) \to \IntV_{\X}\) of graphs in \(\E\).
Observe that there is a morphism of graphs \(H \colon \IntV_0 \to \IntV_1\) defined, on the edges component, via the universal property of the pullback by the arrows \(\Id[\IntV] \colon V_0 \to V_1\) and \((\FreeCatUnit[\IntV_0])_1 \colon V_0 \to {\FreeCat(\IntV_0)}_1\).
Intuitively, \(H (v) \DefEq ((v), \Id[\IntV](v) \colon v \to v)\) (where \((v)\) is a list with one entry, \(v\)).
Then, let \(\Comp[\X] \colon \FreeCat\Ind(X) \to \IntV_{\X}\) extend the morphism of graphs
\[
	\Ind(X) \xrightarrow{( \Id(\Ind(X)), \Hom[\IntV] )} \Ind(X) \Pullback[][\IntV_0] \IntV_0 \xrightarrow{\Id(\Ind(X)) \Prod H} \Ind(X) \Pullback[][\IntV_0] \IntV_1 \Iso \IntV_{\X}
\]
that is, intuitively, the functor sending \(x\) to \(\Id[\X](x)\) and \((x_0, x_1)\) to \(\Id[\IntV](\Hom[\X](x_0, x_1))\).
Then \(\Comp[\X](x_0, \dots, x_n)\) is the (iterated) composition
\[
	\Hom[\X](x_{n-1}, x_n) \MonProd \dots \MonProd \Hom[\X](x_0, x_1) \xrightarrow{\Comp[\X]} \Hom[\X](x_0, x_n)
\]
(remember that the composition of \(\X\) is associative and the monoidal associator is an isomorphism, so this operation is uniquely defined up to isomorphism).
By applying the forgetful functor, consider \(\Comp[\X]\) as a mere morphism of graphs in \(\E\).
Let then \(\MCX_1 \colon \FreeCat\Ind(X) \to \IntV_1\) be the composition of \(\Comp[\X]\) with the morphism of graphs \(\pi_2 \colon \IntV_{\X} \to \IntV_1\) in \(\E\).

It is a matter of routine calculations to prove that the axioms for enriched generalized multicategories are satisfied by the \(\MCX\) so constructed.
We can then state the following proposition.

\begin{proposition}\label{prop:iec-to-egm}
If \(\X\) is a \(\IntV\)-category, then \(\MCX\) is a \(\MCV\)-enriched \(\IdMon\)-multicategory.
\end{proposition}

We shall now give the definition of functor of \(\MCV\)-enriched \(\IdMon\)-multicategories in an informal style.
For more details about the notion of natural transformation of functors of \(T\)-multicategories and the notion of functor of \(\MCV\)-enriched \(\IdMon\)-multicategories, we refer to \textcite{leinster1999generalized}.

\begin{definition}[{Functor of \(\MCV\)-enriched \(\IdMon\)-multicategories \autocite[{Definition~1.3.2}]{leinster1999generalized}}]
Let \(\MCX\) and \(\MCY\) be \(\MCV\)-enriched \(\IdMon\)-multicategories.
A \Def{functor} \(F \colon \MCX \to \MCY\) is given by an object component \(F_0 \colon X \to Y\) and a morphism component \(F_1 \colon \Ind(X) \to \IntV_1\) such that, intuitively,
\(F_1(x_0, x_1) \colon \MCX_0(x_0, x_1) \to \MCY_0(F_0(x_0), F_0(x_1))\)
and
\[
	\begin{tikzcd}[column sep = 0]
		\MCX_0(x_0, x_1)
			\ar[d, "{F_1(x_0, x_1)}"]
		& {}
			\ar[d, phantom, "{\MonProd \dots \MonProd}" description]
		& \MCX_0(x_{n-1}, x_n)
			\ar[d, "{F_1(x_{n-1}, x_n)}"] \\
		\MCY_0(F_0(x_0), F_0(x_1))
		& {}
			\ar[d, "{\MCComp[\MCY]}"]
		& \MCY_0(F_0(x_{n-1}), F_0(x_n)) \\
		& \MCY_0(F_0(x_0), F_0(x_n))
	\end{tikzcd}
\]
is equal to
\[
\MCX_0(x_0, x_1) \MonProd \dots \MonProd \MCX_0(x_{n-1}, x_n)
	\xrightarrow{\MCComp[\MCY]}
\MCX_0(x_0, x_n)
	\xrightarrow{F_1(x_{0}, x_n)}
\MCY_0(F_0(x_{0}), F_0(x_n)).
\]
\end{definition}

Let \(\MCat[\IdMon][\MCV]\) be the category of \(\MCV\)-enriched \(\IdMon\)-multicategories and their functors.
Then we can thus extend the construction from \Cref{prop:iec-to-egm} to a functor.

\begin{proposition}\label{prop:VCat-to-VIdMulticat}
Then there is a functor \(\Cat[\E][\IntV] \to \MCat[\IdMon][\MCV]\).
\end{proposition}
\begin{proof}
Given a functor \(F \colon \X \to \Y\) of \(\IntV\)-enriched categories, there is a \(\MCV\)-enriched \(\IdMon\)-multicategory functor \(\MCX \to \MCY\) given by \(F_0 \colon X_0 \to Y_0\), and \(F_1 \colon \Ind(X) \to \IntV_1\) induced by \(F_1 \colon X \Prod X \to V_1\).
\end{proof}

We now consider the opposite construction, that of the internally enriched category induced by an enriched generalized multicategory.
Let \(\MCX \colon \IndP[X] \to \MCV\) be a \(\MCV\)-enriched \(\IdMon\)-multicategory as in \Cref{eqn:M-enrchd-id-cat}.
Then there is a \(\IntV\)-category \(\X\) on \(X\) whose hom is \(\Hom[\X] = {(\MCX_0)}_1 \colon X \Prod X = {\Ind(X)}_1~\to~V_0\), whose composition is
\[
	\begin{tikzcd}
		X \Prod X \Prod X \Iso {\Ind(X)}_1 \Pullback[\Source[\Ind(X)]][\Target[\Ind(X)]] {\Ind(X)}_1
			\ar[d, "{{(\FreeCatUnit[\Ind(X)])}_1 \Prod {(\FreeCatUnit[\Ind(X)])}_1}"] \\
		{(\FreeCat \Ind(X))}_1 \Pullback[\Source[\FreeCat \Ind(X)]][\Target[\FreeCat \Ind(X)]] {(\FreeCat \Ind(X))}_1
			\ar[d, "{\Comp[\FreeCat \Ind(X)]}"] \\
		{(\FreeCat \Ind(X))}_1
			\ar[d, "{{(\MCX_1)}_1}"] \\
		{(\IntV_1)}_1
			\ar[d, "{{(\pi_2)}_1}"] \\
		V_1
	\end{tikzcd}
\]
and whose identity \(\Id[\X]\) is
\[
		X =
		{\Ind(X)}_0
			\xrightarrow{{(\FreeCatUnit)}_{0}}
		{(\FreeCat\Ind(X))}_0
			\xrightarrow{\MCId[\FreeCat\Ind(X)]}
		{(\FreeCat\Ind(X))}_1
			\xrightarrow{{(\MCX_1)}_{1}}
		{(\IntV_1)}_1
			\xrightarrow{{(\pi_2)}_1}
		V_1.
\]

Given a functor \(F \colon \MCX \to \MCY\) of \(\MCV\)-enriched \(\IdMon\)-multicategory, there is a functor of \(\IntV\)-categories given by \(F_0 \colon X_0 \to Y_0\), and \(F_1 \colon X \Prod X \to V_1\) induced by \(F_1 \colon \Ind(X) \to \IntV_1\).

Thus, we have the converse of \Cref{prop:VCat-to-VIdMulticat}.

\begin{proposition}\label{prop:VIdMulticat-to-VCat}
There is a functor \(\MCat[\IdMon][\MCV] \to \Cat[\E][\IntV]\).
\end{proposition}

It is then a matter of routine calculations to show that the constructions defined in the previous sections are mutually inverse (up to isomorphism), meaning that \(\MCV\)-enriched \(\IdMon\)-multicategories and \(\IntV\)-enriched categories are equivalent notions.

\begin{proposition}
The categories \(\MCat[\IdMon][\MCV]\) and \(\Cat[\E][\IntV]\) are equivalent.
\end{proposition}

As a final corollary of the results of this paper, we can state the relationship between enriched generalized multicategories and enriched indexed categories.

\begin{corollary}
The categories \(\MCat[\IdMon][\MCV]\) and \(\SCat[\E][\Externalization{\IntV}]\) are equivalent.
\end{corollary}

Were there a notion of natural transformation for enriched generalized multicategories making \(\Cat[\E][\IntV]\) and \(\MCat[\IdMon][\MCV]\) equivalent as 2-categories, we would also be able to deduce that \(\MCat[\IdMon][\MCV]\) is a full sub-2-category of \(\ICat[\E][\Externalization{\IntV}]\) as a corollary of \Cref{prop:V-cat-index-cat-fullsubcat}.
Unfortunately, such a notion seems to be absent from our sources, notably \textcite{leinster1999generalized,Leinster02GeneralizedEnrichment}.
We believe such a notion should exist and verify the above property, but such investigations are beyond the scope of this paper, so we leave it as an open question.

\addsec{Acknowledgments}

We wish to thank the Cambridge Trust and the EPSRC, for generously funding the doctoral research resulting in the material contained in this paper.
We also wish to acknowledge and thank our doctoral supervisor, Professor Martin Hyland, for the supervision and the guidance offered on mathematical and academic matters in the course of said research activity;
Professor Giuseppe Rosolini, for the original suggestion to enrich in the internal category of modest sets which inspired this work;
and Dr.~Alessio D'Alì, for kindly providing valuable feedback and suggestions on the paper's draft.

\printbibliography
\end{document}

%% file: math.tex
\theoremstyle{plain}
\newtheorem{theorem}{Theorem}[section]
\newtheorem{proposition}[theorem]{Proposition}
\newtheorem{corollary}[theorem]{Corollary}

\theoremstyle{definition}
\newtheorem{definition}[theorem]{Definition}
\newtheorem{notation}[theorem]{Notation}

\theoremstyle{remark}
\newtheorem{example}[theorem]{Example}
\newtheorem{remark}[theorem]{Remark}

\crefname{axiom}{axiom}{axioms}
\crefname{example}{example}{examples}
\crefname{formula}{formula}{formulae}

\NewDocumentCommand\CatStyle{ m }{\mathcal{#1}}
\NewDocumentCommand\IntCatStyle{ m }{\mathbf{#1}}
\NewDocumentCommand\MultiCatStyle{ m }{\mathfrak{#1}}

\DeclareMathOperator{\CatOp}{Cat}
\NewDocumentCommand{\Cat}		{ O{}O{} }	{\ensuremath{{#2}\CatOp_{#1}}}
\DeclareMathOperator{\MonCatOp}{Mon}
\NewDocumentCommand{\MonCat}	{ O{} }		{\ensuremath{\MonCatOp_{#1}}}
\DeclareMathOperator{\SymMonCatOp}{SymMonCat}
\NewDocumentCommand{\SymMonCat}	{ O{} }		{\ensuremath{\SymMonCatOp_{#1}}}
\DeclareMathOperator{\MulticatOp}{Multicat}
\NewDocumentCommand{\Multicat}	{ O{} }		{\ensuremath{\MulticatOp_{#1}}}
\DeclareMathOperator{\GraphOp}{Graph}
\NewDocumentCommand{\Graph}		{ O{} }		{\ensuremath{\GraphOp_{#1}}}

\DeclareMathOperator{\SetCat}{Set}

\DeclareMathOperator{\ICatOp}{ICat}
\NewDocumentCommand{\ICat}		{ O{}O{} }	{\ensuremath{{#2}\ICatOp_{#1}}}
\DeclareMathOperator{\SCatOp}{SCat}
\NewDocumentCommand{\SCat}		{ O{}O{} }	{\ensuremath{{#2}\SCatOp_{#1}}}
\DeclareMathOperator{\IdOp}					{id}
\NewDocumentCommand\Id			{ O{} }		{\IdOp_{#1}}
\NewDocumentCommand\Source		{ O{} }		{\mathrm{s}_{#1}}
\NewDocumentCommand\Target		{ O{} }		{\mathrm{t}_{#1}}
\DeclareMathOperator\HomOp 					{Hom}
\NewDocumentCommand\Hom			{ O{} } 	{\HomOp_{#1}}
\NewDocumentCommand\Comp		{ O{} }		{\mathbin{\circ_{#1}}}
\NewDocumentCommand\VComp		{ O{} }		{}
\NewDocumentCommand\HComp		{ O{} }		{}
\NewDocumentCommand\MonProd		{O{}O{}}	{\mathbin{\otimes_{#1}}}
\NewDocumentCommand\MonUnit		{ O{} }		{\mathbb{I}_{#1}}
\NewDocumentCommand\MonAssoc	{ O{} }		{\alpha_{#1}}
\NewDocumentCommand\MonUnitorL	{ O{} }		{\lambda_{#1}}
\NewDocumentCommand\MonUnitorR	{ O{} }		{\rho_{#1}}
\NewDocumentCommand\MonSym		{ O{} }		{\sigma_{#1}}
\NewDocumentCommand\Prod		{ O{} } 	{\mathbin{\times_{#1}}}
\NewDocumentCommand\Terminal	{ O{} } 	{\mathbb{1}_{#1}}
\NewDocumentCommand\Pullback 	{O{}O{}}	{\mathbin{{{\leavevmode\vphantom{\Prod}}_{#1}}{\Prod}_{#2}}}
\NewDocumentCommand\Diag		{ O{} } 	{\delta_{#1}}
\NewDocumentCommand\Aug			{ O{} } 	{\mathrm{e}_{#1}}
\NewDocumentCommand\Inv			{ O{} } 	{\mathrm{i}_{#1}}
\NewDocumentCommand\Adjoint		{}			{\dashv}
\DeclareMathOperator\Op						{op}

\DeclareMathOperator{\Ind}					{ind} %
\DeclareMathOperator{\Dis}					{dis} %

\NewDocumentCommand{\TiC}	{}{\vdash}

\DeclarePairedDelimiter{\ExternalizationBrk}{\lbrack}{\rbrack}
\NewDocumentCommand{\Externalization} {m O{}} {\ExternalizationBrk{#1}^{#2}}

\NewDocumentCommand{\intCat}	{O{}}	{\ensuremath{\Cat_{#1}}}

\DeclarePairedDelimiterXPP{\nStEq}[3]{}{\lbrack}{\rbrack}{}{#1 =_{#3} #2}

\NewDocumentCommand{\Eval}		{O{}}	{\mathrm{e}_{#1}}
\DeclarePairedDelimiterX{\PowObj}[2]{\lbrack}{\rbrack}{#1, #2}
\DeclarePairedDelimiterX{\WLim}[2]{\lbrace}{\rbrace}{#2, #1}

\DeclarePairedDelimiterX{\pairEnc}[2]{\langle}{\rangle}{#1, #2}
\DeclarePairedDelimiterX{\PSh}[2]{\lbrack}{\rbrack}{#1, #2}

\DeclareMathOperator{\MCatOp}{MCat}
\NewDocumentCommand{\MCat}		{ O{}O{} }	{\ensuremath{{#2}\MCatOp_{#1}}}
\NewDocumentCommand\MC{ m }{\MultiCatStyle{#1}}

\DeclareMathOperator{\IdMonOp}	{ID} %
\NewDocumentCommand{\IdMon}		{ O{} } {\IdMonOp} %
\NewDocumentCommand{\FreeCatUnit}{ O{} } {\eta_{#1}}
\NewDocumentCommand{\FreeCatMult}{ O{} } {\bullet_{#1}}
\NewDocumentCommand{\Unit} {} {\eta}
\NewDocumentCommand{\Counit} {} {\epsilon}
\DeclareMathOperator{\FreeCat} {fc}
\NewDocumentCommand{\IndP}		{ O{} } {{\Ind(#1)}^{+}}
\NewDocumentCommand\MCId{ O{} } {\mathrm{i}_{#1}}
\NewDocumentCommand\MCSource{ O{} } {\sigma_{#1}}
\NewDocumentCommand\MCTarget{ O{} } {\tau_{#1}}
\NewDocumentCommand\MCComp{ O{} } {\cdot_{#1}}

\NewDocumentCommand \Def			{ m }	{\emph{#1}}

\NewDocumentCommand \Iso			{ }		{\mathrel{\cong}}
\NewDocumentCommand \Entry			{ O{} } {\mathhyphen_{#1}}
\NewDocumentCommand \MathDash		{ }		{\mathrm{-}}

\NewDocumentCommand{\DefEq} {} {\vcentcolon=}
\DeclarePairedDelimiterX{\pair}[2]{\lparen}{\rparen}{#1, #2}

\NewDocumentCommand{\N} {} {\mathbb{N}}